\documentclass{amsart}
\usepackage{amssymb}

\begin{document}
\newtheorem{theorem}{Theorem}
\newtheorem{lemma}[theorem]{Lemma}
\newtheorem{sublemma}{Sublemma}
\newtheorem{proposition}[theorem]{Proposition}
\newtheorem{corollary}[theorem]{Corollary}
\renewcommand{\thefootnote}{\fnsymbol{footnote}}

\newcommand{\myfbox}{\fbox}  \renewcommand{\myfbox}{\relax}

\title[Prime knots with arc index up to 11]%
{Prime knots with arc index up to 11   and an upper bound of arc index for non-alternating knots}

\author{Gyo Taek Jin and Wang Keun Park}

\address{Department of Mathematical Sciences,
Korea Advanced Institute of Science and Technology,
Daejeon, 305-701, Korea}
\email{trefoil@kaist.ac.kr, lg9004@hotmail.com}

\begin{abstract}
Every knot can be embedded in the union of finitely many half planes with a common boundary line in such a way that the portion of the knot in each half plane is a properly embedded arc. The minimal number of such half planes is called the arc index of the knot. We have identified all prime knots with arc index up to 11.
We also proved that the crossing number is an upperbound of arc index for non-alternating knots. As a result the arc index is determined for prime knots up to twelve crossings.
\end{abstract}



\maketitle

\section{Arc presentations and Arc index}
In his foundational work~\cite{C1995} on the \emph{arc index\/} of knots and links, Cromwell showed that every link diagram is isotopic to a diagram which is a finite union of the following local diagrams in such a way that no more than two corners exist in any vertical line and any horizontal line.
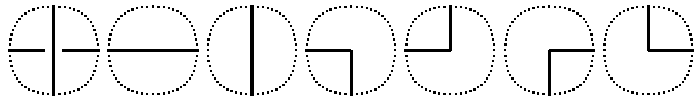
\begin{figure}[htb]
\setlength{\unitlength}{0.3mm}
\newcommand{\dotbox}{
\qbezier[15](20,0.5)(0.5,0.5)(0.5,20)
\qbezier[15](20,0.5)(39.5,0.5)(39.5,20)
\qbezier[15](20,39.5)(0.5,39.5)(0.5,20)
\qbezier[15](20,39.5)(39.5,39.5)(39.5,20) \thicklines }
\centering
\begin{picture}(40,50)(0,-5)
\dotbox \put(20,0){\line(0,1){40}} \put(0,20){\line(1,0){16}}
\put(24,20){\line(1,0){16}}
\end{picture}
\begin{picture}(40,50)(0,-5)
\dotbox \put(0,20){\line(1,0){40}}
\end{picture}
\begin{picture}(40,50)(0,-5)
\dotbox \put(20,0){\line(0,1){40}}
\end{picture}
\begin{picture}(40,50)(0,-5)
\dotbox \put(0,20){\line(1,0){20}}\put(20,20){\line(0,-1){20}}
\end{picture}
\begin{picture}(40,50)(0,-5)
\dotbox \put(20,40){\line(0,-1){20}}\put(20,20){\line(-1,0){20}}
\end{picture}
\begin{picture}(40,50)(0,-5)
\dotbox \put(40,20){\line(-1,0){20}}\put(20,20){\line(0,-1){20}}
\end{picture}
\begin{picture}(40,50)(0,-5)
\dotbox \put(20,40){\line(0,-1){20}}\put(20,20){\line(1,0){20}}
\end{picture}
\caption{Local parts of diagrams}
\end{figure}
Such a diagram is called a \emph{{grid diagram}}. 

\begin{figure}[htb]
\centering
\setlength{\unitlength}{0.4mm}
{
\begin{picture}(40,40)(-20,-20)
\thicklines
\qbezier(-3.449295881,19.56192902)(0,20.45206417)(3.449295881,19.56192902)
\qbezier(3.449295881,19.56192902)(6.831379890,18.68913875)(9.414213562,16.30589621)
\qbezier(9.414213562,16.30589621)(11.83980810,14.06774198)(13.04977866,10.95006445)
\qbezier(14.06224055,5.118236976)(14.01090856,2.605775951)(13.17157288,0.)
\qbezier(13.17157288,0.)(12.54933695,-1.931774688)(11.40438845,-4.150857928)
\qbezier(11.40438845,-4.150857928)(10.43699928,-6.025804354)(9.296942644,-7.801061140)
\qbezier(9.296942644,-7.801061140)(7.947634440,-9.902157238)(6.585786438,-11.40691672)
\qbezier(6.585786438,-11.40691672)(4.748786110,-13.43669072)(2.598597021,-14.73737604)
\qbezier(2.598597021,-14.73737604)(-0.021785052,-16.32248896)(-2.958144665,-16.77647206)
\qbezier(-2.958144665,-16.77647206)(-6.263117918,-17.28744558)(-9.414213562,-16.30589621)
\qbezier(-9.414213562,-16.30589621)(-12.76957900,-15.26071790)(-15.21647954,-12.76814236)
\qbezier(-15.21647954,-12.76814236)(-17.71200714,-10.22603206)(-18.66577542,-6.793786640)
\qbezier(-18.66577542,-6.793786640)(-19.60095887,-3.428420851)(-18.82842712,0.)
\qbezier(-18.82842712,0.)(-18.10292597,3.219703603)(-16.00792332,5.826407592)
\qbezier(-11.46364353,9.619139052)(-9.262122459,10.83091476)(-6.585786438,11.40691672)
\qbezier(-6.585786438,11.40691672)(-4.601702523,11.83393194)(-2.107445805,11.95191908)
\qbezier(-2.107445805,11.95191908)(0,12.05160870)(2.107445805,11.95191908)
\qbezier(2.107445805,11.95191908)(4.601702528,11.83393194)(6.585786438,11.40691672)
\qbezier(6.585786438,11.40691672)(9.262122473,10.83091476)(11.46364353,9.619139052)
\qbezier(11.46364353,9.619139052)(14.14658260,8.142378075)(16.00792332,5.826407592)
\qbezier(16.00792332,5.826407592)(18.10292598,3.219703573)(18.82842712,0.)
\qbezier(18.82842712,0.)(19.60095888,-3.428420879)(18.66577542,-6.793786640)
\qbezier(18.66577542,-6.793786640)(17.71200713,-10.22603208)(15.21647954,-12.76814236)
\qbezier(15.21647954,-12.76814236)(12.76957898,-15.26071791)(9.414213562,-16.30589621)
\qbezier(9.414213562,-16.30589621)(6.263117879,-17.28744558)(2.958144665,-16.77647206)
\qbezier(-2.598597021,-14.73737604)(-4.748786134,-13.43669069)(-6.585786438,-11.40691672)
\qbezier(-6.585786438,-11.40691672)(-7.947634488,-9.902157183)(-9.296942644,-7.801061140)
\qbezier(-9.296942644,-7.801061140)(-10.43699932,-6.025804267)(-11.40438845,-4.150857928)
\qbezier(-11.40438845,-4.150857928)(-12.54933697,-1.931774655)(-13.17157288,0.)
\qbezier(-13.17157288,0.)(-14.01090857,2.605775996)(-14.06224055,5.118236976)
\qbezier(-14.06224055,5.118236976)(-14.12479757,8.180110881)(-13.04977866,10.95006445)
\qbezier(-13.04977866,10.95006445)(-11.83980808,14.06774200)(-9.414213562,16.30589621)
\qbezier(-9.414213562,16.30589621)(-6.831379864,18.68913877)(-3.449295881,19.56192902)
\end{picture}
} \qquad\qquad
{
\begin{picture}(40,40)
\thicklines \put(0,0){\line(1,0){30}} \put(30,0){\line(0,1){30}}
\put(30,30){\line(-1,0){7}} \put(17,30){\line(-1,0){7}}
\put(10,30){\line(0,-1){20}} \put(10,10){\line(1,0){17}}
\put(33,10){\line(1,0){7}} \put(40,10){\line(0,1){30}}
\put(40,40){\line(-1,0){20}} \put(20,40){\line(0,-1){20}}
\put(20,20){\line(-1,0){7}} \put(7,20){\line(-1,0){7}}
\put(0,20){\line(0,-1){20}}
\end{picture}
}
\caption{Grid diagram of a trefoil knot}
\end{figure}
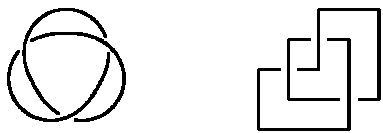

An {$n\times n$} matrix each of whose rows and columns has exactly two {1}'s and {0}'s elsewhere
is called a \emph{{Cromwell matrix}}. By joining two {1}'s in each column of a Cromwell matrix 
with a vertical line segment and
two {1}'s in each row 
with a horizontal line segment which underpasses any vertical line segments that
it crosses, we obtain its grid diagram. 
Conversely, given a grid diagram with $n$ horizontal lines and $n$ vertical lines, we place 1's at each corner and 0's at other points where the lines and their extensions cross, to construct its Cromwell matrix.

\begin{figure}[htb]
\centering
\setlength{\unitlength}{0.11mm}
\myfbox{
\begin{picture}(240,240)(-20,-20)
\thicklines
\put(  0, 80){\line( 0,1){120}}
\put(  0,200){\line( 1,0){160}}
\put(  0, 80){\line( 1,0){ 30}}
\put( 40, 40){\line( 0,1){ 80}}
\put( 40,120){\line( 1,0){ 30}}
\put( 40, 40){\line( 1,0){ 70}}
\put( 80, 80){\line( 0,1){ 80}}
\put( 80,160){\line( 1,0){ 70}}
\put( 80, 80){\line(-1,0){ 30}}
\put(120,  0){\line( 0,1){120}}
\put(120,120){\line(-1,0){ 30}}
\put(120,  0){\line( 1,0){ 80}}
\put(160, 40){\line( 0,1){160}}
\put(160, 40){\line(-1,0){ 30}}
\put(200,  0){\line( 0,1){160}}
\put(200,160){\line(-1,0){ 30}}
\end{picture}
}
\qquad
\myfbox{
\begin{picture}(240,240)(-20,-20)
\qbezier[40](-20,  0)(100,  0)(220,  0)
\qbezier[40](-20, 40)(100, 40)(220, 40)
\qbezier[40](-20, 80)(100, 80)(220, 80)
\qbezier[40](-20,120)(100,120)(220,120)
\qbezier[40](-20,160)(100,160)(220,160)
\qbezier[40](-20,200)(100,200)(220,200)
\qbezier[40](  0,-20)(  0,100)(  0,220)
\qbezier[40]( 40,-20)( 40,100)( 40,220)
\qbezier[40]( 80,-20)( 80,100)( 80,220)
\qbezier[40](120,-20)(120,100)(120,220)
\qbezier[40](160,-20)(160,100)(160,220)
\qbezier[40](200,-20)(200,100)(200,220)
\scriptsize
{
\put( -15,204){$1$}\put( 25,204){$0$}\put( 65,204){$0$}\put( 105,204){$0$}\put( 145,204){$1$}\put( 185,204){$0$}
\put( -15,164){$0$}\put( 25,164){$0$}\put( 65,164){$1$}\put( 105,164){$0$}\put( 145,164){$0$}\put( 185,164){$1$}
\put( -15,124){$0$}\put( 25,124){$1$}\put( 65,124){$0$}\put( 105,124){$1$}\put( 145,124){$0$}\put( 185,124){$0$}
\put( -15, 84){$1$}\put( 25, 84){$0$}\put( 65, 84){$1$}\put( 105, 84){$0$}\put( 145, 84){$0$}\put( 185, 84){$0$}
\put( -15, 44){$0$}\put( 25, 44){$1$}\put( 65, 44){$0$}\put( 105, 44){$0$}\put( 145, 44){$1$}\put( 185, 44){$0$}
\put( -15,  4){$0$}\put( 25,  4){$0$}\put( 65,  4){$0$}\put( 105,  4){$1$}\put( 145,  4){$0$}\put( 185,  4){$1$}
}%
{
\thicklines
\put(  0, 80){\line( 0,1){120}}
\put(  0,200){\line( 1,0){160}}
\put(  0, 80){\line( 1,0){ 30}}
\put( 40, 40){\line( 0,1){ 80}}
\put( 40,120){\line( 1,0){ 30}}
\put( 40, 40){\line( 1,0){ 70}}
\put( 80, 80){\line( 0,1){ 80}}
\put( 80,160){\line( 1,0){ 70}}
\put( 80, 80){\line(-1,0){ 30}}
\put(120,  0){\line( 0,1){120}}
\put(120,120){\line(-1,0){ 30}}
\put(120,  0){\line( 1,0){ 80}}
\put(160, 40){\line( 0,1){160}}
\put(160, 40){\line(-1,0){ 30}}
\put(200,  0){\line( 0,1){160}}
\put(200,160){\line(-1,0){ 30}}
}
\end{picture}
}
\qquad
\myfbox{
\begin{picture}(240,240)(-20,-20)
\scriptsize
{
\put( -15,204){$1$}\put( 25,204){$0$}\put( 65,204){$0$}\put( 105,204){$0$}\put( 145,204){$1$}\put( 185,204){$0$}
\put( -15,164){$0$}\put( 25,164){$0$}\put( 65,164){$1$}\put( 105,164){$0$}\put( 145,164){$0$}\put( 185,164){$1$}
\put( -15,124){$0$}\put( 25,124){$1$}\put( 65,124){$0$}\put( 105,124){$1$}\put( 145,124){$0$}\put( 185,124){$0$}
\put( -15, 84){$1$}\put( 25, 84){$0$}\put( 65, 84){$1$}\put( 105, 84){$0$}\put( 145, 84){$0$}\put( 185, 84){$0$}
\put( -15, 44){$0$}\put( 25, 44){$1$}\put( 65, 44){$0$}\put( 105, 44){$0$}\put( 145, 44){$1$}\put( 185, 44){$0$}
\put( -15,  4){$0$}\put( 25,  4){$0$}\put( 65,  4){$0$}\put( 105,  4){$1$}\put( 145,  4){$0$}\put( 185,  4){$1$}
}%
\end{picture}
}
\caption{Construction of Cromwell matrix from a grid diagram and its inverse}
\end{figure}
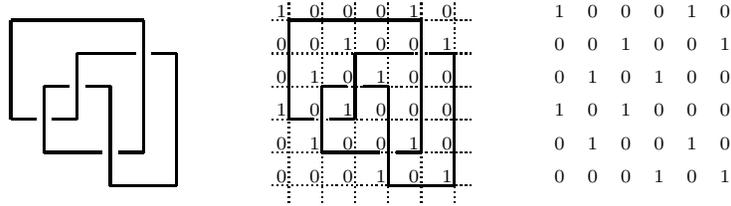

\begin{proposition}[Cromwell 1995]
Every link admits an arc presentation.
\end{proposition}

\begin{proof} As illustrated in Figure~\ref{fig:diagram2arc}, the horizontal line segments of a grid diagram are horizontally pulled backwards to touch a vertical axis behind the diagram to form an arc presentation. As every link has a grid diagram, we can conclude that every link has an arc presentation.
\end{proof}

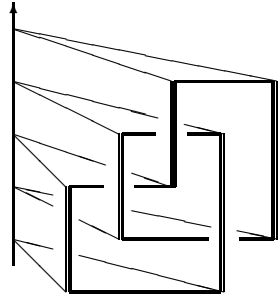
\begin{figure}[h]
\centering
\setlength{\unitlength}{0.7mm}
\begin{picture}(50,55)(-10,0)
{
\put(-10,10){\line(1,-1){10}}
\put(-10,10){\line(4,-1){7.5}} \put(2.5,6.875){\line(4,-1){27.5}}
\put(-10,20){\line(2,-1){7.5}} \put(2.5,13.75){\line(2,-1){7.5}}
\put(-10,20){\line(5,-1){7.5}} \put(2.5,17.6){\line(5,-1){5}} \put(12.5,15.4){\line(5,-1){15}} \put(32.5,11.6){\line(5,-1){7.5}}
\put(-10,30){\line(1,-1){10}}
\put(-10,30){\line(3,-1){17.5}} \put(12.5,22.5){\line(3,-1){7.5}}
\put(-10,40){\line(2,-1){20}}
\put(-10,40){\line(4,-1){27.5}} \put(22.5,31.875){\line(4,-1){7.5}}
\put(-10,50){\line(5,-1){50}}
\put(-10,50){\line(3,-1){30}}
\put(30,0){\line(0,1){30}}
\put(10,30){\line(0,-1){20}}
\put(40,10){\line(0,1){30}}
\put(20,40){\line(0,-1){20}}
\put(0,20){\line(0,-1){20}}
}
\put(-10,5){\vector(0,1){50}}
\linethickness{0.3mm} 
\put(0.7,0){\line(1,0){28.6}}
\put(29.3,30){\line(-1,0){6.3}}
\put(17,30){\line(-1,0){6.3}}
\put(10.7,10){\line(1,0){16.3}}
\put(33,10){\line(1,0){6.3}}
\put(39.3,40){\line(-1,0){18.6}}
\put(20.7,20){\line(-1,0){7.7}}
\put(7,20){\line(-1,0){6.3}}
\put(29.3,0){\line(0,1){30}}
\put(10.7,30){\line(0,-1){20}}
\put(39.3,10){\line(0,1){30}}
\put(20.7,40){\line(0,-1){20}}
\put(0.7,20){\line(0,-1){20}}
\end{picture}
\caption{Construction of an arc presentation from a grid diagram}\label{fig:diagram2arc}
\end{figure}

The minimal number of pages of all arc presentations of a link {$L$} is called the \emph{{arc index}} of
{$L$} and is denoted by {$\alpha(L)$}.
It is known that the trefoil knot is the only link with arc index {5}.
The table below shows all links up to arc index {5}. In~\cite{N1999}, Nutt identified all knots up to arc index 9. In~\cite{B2002}, Beltrami determined arc index for prime knots up to ten crossings. In~\cite{Jin2006}, Jin et al.\ identified all prime knots up to arc index~10. Ng determined arc index for prime knots up to eleven crossings~\cite{Ng2006}. Matsuda determined the arc index for torus knots~\cite{H2006}. The \emph{Table of Knot Invariants\/}~\cite{knotinfo} gives the arc index for prime knots up to 12 crossings.

\begin{table}[htb]\small
\renewcommand\arraystretch{1.2}
\centering
\caption{Links with arc index up to 5} 
{\begin{tabular}{|c|c|c|c|c|}
\hline
$\alpha(L)$ & 2& 3& 4& 5\\ \hline
$L$& unknot & none & 2-component unlink, Hopf link& trefoil\\ \hline
\end{tabular}}
\end{table}

\section{Tabulation of prime knots by arc index}
As there are finitely many {$n\times n$} Cromwell matrices for each {$n\ge2$}, there are finitely many knots with arc
index up to {$n$}, for each {$n\ge2$}, hence, there are only finitely many knots with arc index {$n$}, for each {$n\ge2$}.
By the following theorem of Cromwell, we only need to tabulate arc index of prime knots to determine arc index of all knots.
\begin{theorem}[Cromwell 1995]
If {$K_1$}, {$K_2$} are nontrivial, then
$$
{\alpha(K_1\sharp K_2)=\alpha(K_1)+\alpha(K_2)-2}
$$
\end{theorem}

The following moves on a Cromwell matrix does not change the link type of the corresponding grid diagram up to mirror images.
\begin{enumerate}
\item[M1.] Flipping in a horizontal axis, a vertical axis a diagonal axis and an antidiagonal axis.
\item[M2.] Rotation in the plane by 90 degrees.
\item[M3.] Moving the first row to the bottom and moving the first column to the rear.
\item[M4.] Exchange of two adjacent rows or columns whose ones are in non-interleaving position.
\end{enumerate}

\begin{figure}[h]
\centering
$\small
\begin{bmatrix}
&&&\cdots&&&\\
\bf1& 0& 0& 0& 0& \bf1& 0\\
0& \bf1& 0& 0& \bf1& 0& 0\\
&&&\cdots&&&
\end{bmatrix}
\Longleftrightarrow
\begin{bmatrix}
&&&\cdots&&&\\
0& \bf1& 0& 0& \bf1& 0& 0\\
\bf1& 0& 0& 0& 0& \bf1& 0\\
&&&\cdots&&&
\end{bmatrix}
$
\caption{Interchange of non-interleaved rows}\label{fig:non-interleaved rows}
\end{figure}

\begin{figure}[h]
\centering
\setlength{\unitlength}{0.0105cm}
\begin{picture}(240,260)(-20,-20)
%
\thicklines{
\put(  0, 80){\line(0,1){120}} \put(  0,200){\line( 1,0){160}} \put(  0, 80){\line( 1,0){ 30}}
\put( 40, 40){\line(0,1){ 80}} \put( 40,120){\line( 1,0){ 30}} \put( 40, 40){\line( 1,0){ 70}}
\put( 80, 80){\line(0,1){ 80}} \put( 80,160){\line( 1,0){ 70}} \put( 80, 80){\line(-1,0){ 30}}
\put(120,  0){\line(0,1){120}} \put(120,120){\line(-1,0){ 30}} \put(120,  0){\line( 1,0){ 80}}
\put(160, 40){\line(0,1){160}}                                 \put(160, 40){\line(-1,0){ 30}}
\put(200,  0){\line(0,1){160}} \put(200,160){\line(-1,0){ 30}}
}
\end{picture}
\qquad
\raise32pt\hbox{
${\small
\begin{bmatrix}
\ 1&\ 0&\ 0&\ 0&\ 1&\ 0\ \\
\ 0&\ 0&\ 1&\ 0&\ 0&\ 1\ \\
\ 0&\ 1&\ 0&\ 1&\ 0&\ 0\ \\
\ 1&\ 0&\ 1&\ 0&\ 0&\ 0\ \\
\ 0&\ 1&\ 0&\ 0&\ 1&\ 0\ \\
\ 0&\ 0&\ 0&\ 1&\ 0&\ 1\
\end{bmatrix}
}
$
}
\caption{$100010\ 001001\ 010100\ 101000\ 010010\ 000101_2$}\label{fig:norm}
\end{figure}
The \emph{{norm}} of a Cromwell matrix is the {$n^2$} digit binary number obtained by concatenating its rows.
An example is shown in Figure~\ref{fig:norm}.

To tabulate prime knots with arc index up to 11, we proceeded with the following steps, for each integer $n=5,\ldots,11$.
\begin{enumerate}
\item Generate all {$n\times n$} Cromwell matrices whose leading entry is 1 in the norm-decreasing order.
\item Discard those corresponding to links of more than one components.
\item Discard if its grid diagram is not prime.
\item Discard if a sequence of moves M1--M4 ever increase the norm.
\item Discard if a sequence of move M4 ever makes two 1's adjacent horizontally or vertically, as their existence causes a reduction of the size of Cromwell matrix.
\item Identify the knot of its grid diagram.
\item Discard the knot if it already appeared for a smaller {$n$}.
\end{enumerate}

Our computer program in which the steps (1)--(5) were implemented produced 663,341 Cromwell matrices for $n=11$ and their Dowker-Thistlethwaite codes (`DT codes' for short). Using Knotscape~\cite{knotscape}, we were able to eliminate most of the duplications and obtained 2,721 distinct DT codes of prime knots. They include
\begin{itemize}
\item All prime knots of arc index 11. There are 2,335 such knots.
\item All prime knots up to arc index 10 except $13n_{4639}$. There are 281 such knots~\cite{Jin2006}.
\item All prime knots up to 10 crossings except the 10 crossing alternating knots. There are 126 such knots.
\item Some duplications on knots with more than 16 crossings.
\end{itemize}

\begin{table}[t]
\small
\caption{Prime knots up to arc index 11 or up to 12 crossings}
{
\newcommand{\vsp}{\vrule width0pt height 10pt depth.5pt} 
\newcommand{\unk}{$*$} 
\begin{tabular}{|c|c|c|c|c|c|c|c|c|c|c|c|}
\hline
\vrule width0pt height 12pt depth22pt
\lower15pt\hbox{Crossings} \kern-25pt Arc index &
        \hbox to 7pt{\hfil 5\hfil}&
        \hbox to 7pt{\hfil 6\hfil}&
        \hbox to 7pt{\hfil 7\hfil}&
        \hbox to 7pt{\hfil 8\hfil}&
        \hbox to 7pt{\hfil 9\hfil}&
        \hbox to 7pt{\hfil {10}\hfil}&
        \hbox to 7pt{\hfil {11}\hfil}&
        \hbox to 7pt{\hfil {12}\hfil}&
        \hbox to 7pt{\hfil {13}\hfil}&
        \hbox to 7pt{\hfil {14}\hfil}& Subtotal\\
\hline
\vsp  3&1\footnotemark[1]& & & & &  &  &  &  &  &  1\\
\hline
\vsp  4& &1\footnotemark[1]& & & &  &  &  &  &  &  1\\
\hline
\vsp  5& & &2\footnotemark[1]& & &  &  &  &  &  &  2\\
\hline
\vsp  6& & & &3\footnotemark[1]& &  &  &  &  &  &  3\\
\hline
\vsp  7& & & & &7\footnotemark[1]&  &  &  &  &  &  7\\
\hline
\vsp  8& & &1&2& &{18}\footnotemark[1]&  &  &  &  &21\\
\hline
\vsp  9& & & &2&6&  &{41}\footnotemark[1]&  &  &  & 49\\
\hline
\vsp 10& & & &1&9&32&  &  123\footnotemark[1]&  &  &  165\\
\hline
\vsp 11& & & & &4&46&135&  & 367\footnotemark[1]&  & 552\\
\hline
\vsp 12& & & & &2&48& \bf 211\footnotemark[2]&  \bf 627&  &1288\footnotemark[1]& 2176\\
\hline
\vsp 13& & & & & &49& \bf 399  \\
\cline{1-8}
\vsp 14& & & & & &17& \bf 477  \\
\cline{1-8}
\vsp 15& & & & &1&22& \bf 441  \\
\cline{1-8}
\vsp 16& & & & & & 7& \bf 345  \\
\cline{1-8}
\vsp 17& & & & & & 1& \bf 192  \\
\cline{1-8}
\vsp 18& & & & & &  & \bf 75\footnotemark[3]  \\
\cline{1-8}
\vsp 19& & & & & &  & \bf 12\footnotemark[3]  \\
\cline{1-8}
\vsp 20& & & & & &  & \bf 3\footnotemark[3]  \\
\cline{1-8}
\vsp 21& & & & & &  & \bf 3\footnotemark[3]  \\
\cline{1-8}
\vsp 22& & & & & &  &    \\
\cline{1-8}
\vsp 23& & & & & &  &    \\
\cline{1-8}
\vsp 24& & & & & &  & \bf 1 \\
\cline{1-8}
\cline{1-8}
\vsp Subtotal& 1& 1& 3& 8& {29}& {240}& \bf {2335}\\
\cline{1-8}
\end{tabular}\\
}
\end{table}

\footnotetext[1]{The number of prime alternating knots with the given crossing number.}
\footnotetext[2]{The numbers in boldface are newly determined by this work.}
\footnotetext[3]{Some of these knots may have smaller crossing number but not smaller than $17$.}

To handle the duplications on knots with more than 16 crossings, we used polynomial invariants and hyperbolic invariants built into Knotscape. We also used Knotplot~\cite{knotplot} to isotope a knot for another DT code that may work better in Knotscape. Alexander Stoimenow strained out all the duplicates which we could not handle. The knots of arc index 11 counted in Table~2 are all distinct although the crossing number may change for those of 18 crossings or higher.
The authors wrote a separate article for the list of all prime knots with arc index up to 11 and their minimal arc presentations~\cite{Jin2007}.


\section{Alternating knots and non-alternating knots}

Combining Theorem~\ref{Bae-Park, 2000} and Theorem~\ref{Morton-Beltrami, 1998}, we know that the arc index of alternating links is equal to the minimal crossing number plus 2. The main diagonal of Table~2 shows this fact. The blanks below the main diagonal of Table~2 verifies that the inequality in Theorem~\ref{thm:main} holds up to 12 crossing knots.
Theorem~\ref{Bae-Park, 2000} was conjectured by Cromwell and Nutt~\cite{CN1996}. Beltrami obtained an inequality sharper than the one in Theorem~\ref{thm:main} for semi-alternating links~\cite{B2002}.

\begin{theorem}[Bae-Park]\label{Bae-Park, 2000}
Let $L$ be any prime link and let $c(L)$ denote the minimal crossing number of $L$. Then
$\alpha(L)\le c(L)+2$.
\end{theorem}

\begin{theorem}[Morton-Beltrami]\label{Morton-Beltrami, 1998}
Let $L$ be an alternating link with the minimal crossing number $c(L)$, then $\alpha(L)\ge c(L)+2$.
\end{theorem}

\begin{theorem}\label{thm:main}
A prime link {$L$} is non-alternating if and only if $$\alpha(L)\le c(L).$$
\end{theorem}

This theorem allows us to put 627 in Table~2 for the number of 12 crossing knots with arc index 12.

\section{Proof of Theorem~\ref{thm:main}}
\newcounter{enumicnt}
A \emph{knot-spoke diagram} $D$ is a finite connected plane graph with the following properties.
\begin{enumerate}
\item There are three kinds of vertices in $D$; a distinguished vertex $v_0$ with valency at least four, 4-valent vertices, and 1-valent vertices.
\item Every edge incident to a 1-valent vertex is also incident to $v_0$. Such an edge is called a \emph{spoke}.
\setcounter{enumicnt}{\theenumi}
\end{enumerate}
A knot-spoke diagram $D$ is said to be \emph{prime} if no simple closed curve meeting $D$ in two interior points of edges separates multi-valent vertices into two parts.
A multi-valent vertex $v$ of a knot-spoke diagram $D$ is said to be a \emph{cut-point} if there is a simple closed curve $S$ meeting $D$ in $v$ and separating non-spoke edges into two parts.
A cut-point free knot-spoke diagram with more than one non-spoke edges cannot have a \emph{loop}\footnote{A closed curve created by a single edge}.
A loop in $D$ is said to be \emph{simple\/} if the other non-spoke edges are in one side.
If a prime knot-spoke diagram $D$ has a cut-point, then the distinguished vertex $v_0$ must be the cut-point with valency bigger than four.
A knot-spoke diagram without any non-spoke edges is called a \emph{wheel diagram}.

The valency of the distinguished vertex $v_0$ is an even number plus the number of spokes.
To obtain the type of a knot or link which can be projected onto a knot-spoke diagram $D$, we may assign relative heights of the endpoints of edges of $D$ in the following way.
\begin{enumerate}
\addtocounter{enumi}{\theenumicnt}
\item At every 4-valent vertex, pairs of opposite edges meet in two distinct levels. Their relative heights are indicated by drawing a small neighborhood of the vertex as a crossing in a knot diagram.
\item If the distinguished vertex $v_0$ is incident to $2a$ non-spoke edges and $b$ spokes, then its small neighborhood  is the projection of $n=a+b$ arcs at distinct levels whose relatives heights can be specified by the numbers $1,\cdots,n$. Every spoke is understood as the projection of an arc on a vertical plane whose endpoints project to $v_0$.
\setcounter{enumicnt}{\theenumi}
\end{enumerate}

Suppose a knot-spoke diagram $D$ has height information at multi-valent vertices, then
\begin{enumerate}
\addtocounter{enumi}{\theenumicnt}
\item $D$ determines a knot(or link) $L$. In this case, we say that $D$ is a knot-spoke diagram of $L$.
\item If $D$ has only 4-valent vertices then it is a knot(or link) diagram.
\item If $D$ has no non-spoke edges, it is an arc presentation.
\setcounter{enumicnt}{\theenumi}
\end{enumerate}

In \cite{BP2000}, Bae and Park\footnote{Not the second author of this article} used knot-spoke diagrams to prove Theorem~\ref{Bae-Park, 2000}.

\medskip 

Let $e$ be an edge of a cut-point free knot-spoke diagram $D$ incident to $v_0$
and to another vertex $v_1$ which is 4-valent. We denote by $D_e$ the knot-spoke diagram
obtained by contracting $e$ and replacing any simple loop thus
created by a spoke in the following way:

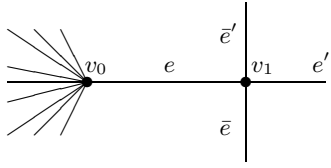
\begin{figure}[htb]
\centering
\myfbox{
\begin{picture}(120,60)
\small
\put(30,30){\circle*{4}} \put(90,30){\circle*{4}}
\put(29,34){$v_0$} \put(59,34){$e$} \put(92,34){$v_1$}
\put(80,10){$\bar e$} \put(80,45){$\bar e^\prime$} \put(115,34){$e^\prime$}
\put(30,30){\line(-1, 2){10}}
\put(30,30){\line(-1, 1){20}}
\put(30,30){\line(-3, 2){30}}
\put(30,30){\line(-5, 1){30}}
\put(30,30){\line(-1, 0){30}}
\put(30,30){\line(-4,-1){30}}
\put(30,30){\line(-3,-2){30}}
\put(30,30){\line(-1,-1){20}}
\put(30,30){\line(-1,-2){10}}
\put(30,30){\line( 1, 0){60}}
\put(90,30){\line(1, 0){30}}
\put(90,30){\line(0, 1){30}}
\put(90,30){\line(0,-1){30}}
\end{picture}
}
\caption{Local diagram of $D$ near $e$}
\label{fig:D near e}
\end{figure}

\medskip
Suppose that $e$, $\bar e$, $e^\prime$ and $\bar e^\prime$ are four edges incident to $v_1$, around in this order so that $v_1$ is the crossing of the two arcs $e\cup e^\prime$ and $\bar e\cup\bar e^\prime$.
Before the contraction of $e$, we adjust the endpoint-heights at $v_1$ as follows.
\begin{enumerate}
\addtocounter{enumi}{\theenumicnt}
\item $e$ is made horizontal so that $e^\prime$ is given the height of $e$ at $v_0$.
\item If $\bar e\cup\bar e^\prime$ undercrosses $e\cup e^\prime$ then $\bar e$ and $\bar e^\prime$ are given a height lower than the lowest height at $v_0$. If $\bar e\cup\bar e^\prime$ overcrosses $e\cup e^\prime$ then $\bar e$ and $\bar e^\prime$ are given a height higher than the highest height at $v_0$.
\setcounter{enumicnt}{\theenumi}
\end{enumerate}
Then as $e$ is contracted to $v_0$, the ends of $\bar e$, $e^\prime$ and $\bar e^\prime$ at $v_1$ are moved to $v_0$ horizontally along $e$. If none of $\bar e$, $e^\prime$ and $\bar e^\prime$ is incident to $v_0$ in $D$, no loop is created by the contraction of $e$. In this case, $D_e$ is obtained by the contraction of $e$ only. Suppose $\bar e$ is incident to $v_0$ in $D$. Because $D$ is cut-point free, $\bar e$ becomes a simple loop by the contraction of $e$. To replace this loop by a spoke, we fold the half of $\bar e$ previously incident to $v_1$ onto the other half. If in case $\bar e^\prime$ is incident to $v_0$ in $D$, a simple loop is created and is replaced by a spoke in a similar manner. If in case $e^\prime$ is incident to $v_0$ in $D$, a non-simple loop is created. This loop remains as is in~$D_e$ so that $v_0$ becomes a cut-point unless the loop is the only non-spoke edge of $D_e$.

\medskip
There are two important facts to point out.
\begin{enumerate}
\addtocounter{enumi}{\theenumicnt}
\item Under the process of of creating $D_e$ from $D$, the sum of the number of regions divided by the diagram and the number of spokes is unchanged.
\item $D_e$ is prime if $D$ is prime.
\setcounter{enumicnt}{\theenumi}
\end{enumerate}

\begin{lemma}\label{lem:adapt-key-BP2000} Let $D$ be a knot-spoke diagram without cut-points.
Suppose that $D$ has at least two multi-valent vertices. Then there
are at least two non-loop non-spoke edges $e$ and $f$, incident to
$v_0$, such that the knot-spoke diagrams $D_e$ and $D_f$ have no cut-points.
\end{lemma}

This lemma is adapted from Lemma~\ref{lem:key-BP2000}~\cite[Lemma 1]{BP2000} to apply directly to knot-spoke diagrams.
We omit the proof as this will be almost a direct translation into the language of knot-spoke diagrams.

\begin{lemma}[Bae-Park]\label{lem:key-BP2000} Let $G$ be a connected plane
graph without cut vertices and $v$ a vertex of $G$. For an edge $e$
incident to $v$, let $\overline{G_e}$ denote the graph obtained from
$G$ by contracting $e$ and then by deleting all the innermost loops
based at $v$. Suppose that $G$ has at least two vertices and all
vertices except $v$ are 4-valent. Then there are at least two edges
$e$ and $f$, incident to $v$, such that new graphs $\overline{G_e}$
and $\overline{G_f}$ have no cut vertices.
\end{lemma}

Suppose that $D$ is a connected link diagram of a link $L$ without nugatory crossings. Let $c(D)$ and $r(D)$ denote the number of crossings in $D$ and the number of regions of the plane divided by $D$, respectively. Then $r(D)=c(D)+2$. In proving Theorem~\ref{Bae-Park, 2000}, the authors of \cite{BP2000} considered $D$ as a knot-spoke diagram and showed that there is a sequence of edges $e_1,\ldots,e_k$ of $D$ such that
$$D_{e_1 \cdots\,e_k}=({}\cdots(D_{e_1})_{e_2}\cdots)_{e_k}$$
 has only one non-spoke edge which then must be a loop. Using the fact
$$s(D_{e_1\cdots\,e_i})+r(D_{e_1\cdots\,e_i})=r(D)$$
for all $i$, where $s(\ )$ denotes the number of spokes, they found that $s(D_{e_1\cdots\,e_k})=r(D)-2=c(D)$. Then the loop which is the only non-spoke edge can be made into two spokes so that the wheel diagram so obtained has exactly $c(D)+2$ spokes, showing $\alpha(L)\le c(D)+2$. The  inequality $\alpha(L)< c(D)+2$ when $D$ is non-alternating was obtained by contracting two adjacent edges simultaneously creating one less spokes than usual~\cite{BP2000,B2002}.
In the following, we argue that such contractions can be done twice, simultaneously or successively, to construct a wheel diagram having $c(D)$ spokes,  showing $\alpha(L)\le c(L)$ when $D$ has the minimal number of crossings.

Before we analyze non-alternating diagrams for our purpose, we state our main tool.
\begin{proposition}\label{prop:main tool}
Let $D$ be a prime cut-point free knot-spoke diagram and let $e$ be an edge incident to $v_0$ and to another $4$-valent vertex $v_1$ such that $D_e$ has a cut-point. Then there exists a simple closed curve $S_e$ satisfying the following conditions.
\begin{enumerate}
\item $D_e\cap S_e=v_0$
\item $S_e$ separates $\bar e$ and $\bar e^\prime$ where the four edges incident to $v_1$ in $D$ are labeled with $e$, $\bar e$, $e^\prime$, $\bar e^\prime$ as in Figure~\ref{fig:D near e}.
\item $S_e$ separates $D_e$ into two knot-spoke diagrams $\bar D$ and $\bar D^\prime$ containing $\bar e$ and $\bar e^\prime$, respectively. Furthermore $\bar D^\prime$ is prime and cut-point free, and there is a sequence of non-spoke edges $e_1,\ldots,e_k$ of $D$ not contained in $\bar D^\prime$ such that the knot-spoke diagram $D_{e_1 e_2 \cdots\,e_k}$ is identical with $\bar D^\prime$ on non-spoke edges in one side of $S_e$ and has only spokes in the other side.
\end{enumerate}
\end{proposition}

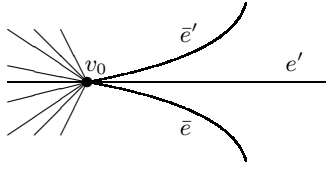
\begin{figure}[htb]
\centering
\myfbox{
\begin{picture}(120,60)
\small
\put(30,30){\circle*{4}}
\put(29,34){$v_0$} \put(105,34){$e^\prime$}
\put(65,10){$\bar e$} \put(65,45){$\bar e^\prime$}
\put(30,30){\line(-1, 2){10}}
\put(30,30){\line(-1, 1){20}}
\put(30,30){\line(-3, 2){30}}
\put(30,30){\line(-5, 1){30}}
\put(30,30){\line(-1, 0){30}}
\put(30,30){\line(-4,-1){30}}
\put(30,30){\line(-3,-2){30}}
\put(30,30){\line(-1,-1){20}}
\put(30,30){\line(-1,-2){10}}
\put(30,30){\line( 1, 0){60}}
\qbezier(30,30)(85,40)(90,60)
\put(90,30){\line(1, 0){30}}
\qbezier(30,30)(85,20)(90, 0)
\end{picture}
}
\caption{Local diagram of $D_e$ near $v_0$}
\label{fig:De near v0}
\end{figure}

\begin{proof}
Because $D$ is cut-point free, $v_0$ must be the only cut-point of $D_e$, hence a simple closed curve $S_e$ satisfying (1) exists.
Then we may consider a simple closed curve $S$ such that $D\cap S=e$ and $S_e$ is obtained from $S$ by contracting $e$. If $S_e$ does not separate $\bar e$ and $\bar e^\prime$ in $D_e$, then $S$ can be pushed off $e$ so that it meets $D$ only in $v_0$. This contradicts the assumption that $D$ is cut-point free. Therefore condition~(2) holds.
Because $D$ is cut-point free, we only need to consider the following three cases to prove condition (3).
\begin{itemize}
\item{} \textbf{\it Case 1.} The $S_e$ satisfying (1) and (2) is essentially unique and $S$ separates $\{\bar e,e^\prime,\bar e^\prime\}$ into $\{\bar e\}$ and $\{e^\prime,\bar e^\prime\}$.
\item{} \textbf{\it Case 2.} The $S_e$ satisfying (1) and (2) is essentially unique and $S$ separates $\{\bar e,e^\prime,\bar e^\prime\}$ into $\{\bar e,e^\prime\}$ and $\{\bar e^\prime\}$.
\item{} \textbf{\it Case 3.} There are two essentially unique distinct simple closed curves $S$ and $S^\prime$ satisfying $D\cap S=D\cap S^\prime=e$ such that $S_e$ and $S^\prime_e$ satisfy (1) and (2), and $S$ and $S^\prime$ separate $\{\bar e,e^\prime,\bar e^\prime\}$ into $\{\bar e\}$, $\{e^\prime,\bar e^\prime\}$ and $\{\bar e,e^\prime\}$, $\{\bar e^\prime\}$, respectively.
\end{itemize}

In the Figures~\ref{fig:case1}--\ref{fig:last in proof} of knot-spoke diagrams, all the spokes are omitted for simplicity.

\begin{figure}[htb]
\centering
\myfbox{
\begin{picture}(110,75)(-10,-15)
\small
\put(30,30){\circle*{4}}
\put(29,34){$v_0$} \put(92,34){$e^\prime$}
\put(65,13){$\bar e$} \put(65,45){$\bar e^\prime$}
\put(30,30){\line(-1, 2){10}}
\put(30,30){\line(-1, 1){20}}
\put(30,30){\line(-3, 2){30}}
\put(30,30){\line(-5, 1){30}}
\put(30,30){\line(-1, 0){30}}
\put(30,30){\line(-4,-1){30}}
\put(30,30){\line(-3,-2){30}}
\put(30,30){\line(-1,-1){20}}
\put(30,30){\line(-1,-2){10}}
\put(30,30){\line( 1, 0){60}}
\qbezier(30,30)(85,40)(90,60)
\put(90,30){\line(1, 0){10}}
\qbezier(30,30)(85,20)(90,10)
\put(-5,10){\line(1,0){100}} \put(95,10){\line(0,-1){20}} \put(-5,-10){\line(1,0){100}} \put(-5,10){\line(0,-1){20}}
\put(43,-4){$T$}
\put(-15,15){$S_e$}
\thicklines
\qbezier[20](30,30)(10,20)(-10,10)
\qbezier[15](-10,10)(-10,-2.5)(-10,-15)
\qbezier[50](-10,-15)(45,-15)(100,-15)
\qbezier[18](100,-15)(100,2.5)(100,20)
\qbezier[35](100,20)(65,25)(30,30)
\end{picture}
}
\qquad\qquad
\myfbox{
\begin{picture}(110,75)(-10,-15)
\small
\put(30,30){\circle*{4}}
\put(29,34){$v_0$} \put(92,34){$e^\prime$}
\put(55,10){$\bar e$} \put(65,45){$\bar e^\prime$}
\put(30,30){\line(-1, 2){10}}
\put(30,30){\line(-1, 1){20}}
\put(30,30){\line(-3, 2){30}}
\put(30,30){\line(-5, 1){30}}
\put(30,30){\line(-1, 0){30}}
\put(30,30){\line(-4,-1){30}}
\put(30,30){\line( 1, 0){60}}
\qbezier(30,30)(85,40)(90,60)
\put(90,30){\line(1, 0){10}}
\qbezier(30,30)(60,20)(50,10) \qbezier(30,30)(40,0)(50,10)
\put(15,-8){$(Tv_0)_{e_1 \cdots\,e_{k-1}}$}
\put(-15,15){$S_e$}
\thicklines
\qbezier[20](30,30)(10,20)(-10,10)
\qbezier[15](-10,10)(-10,-2.5)(-10,-15)
\qbezier[50](-10,-15)(45,-15)(100,-15)
\qbezier[18](100,-15)(100,2.5)(100,20)
\qbezier[35](100,20)(65,25)(30,30)
\end{picture}
}
\caption{Case 1 : $D_e$ and $D_{e\, e_1 \cdots\,e_{k-1}}$}\label{fig:case1}
\end{figure}
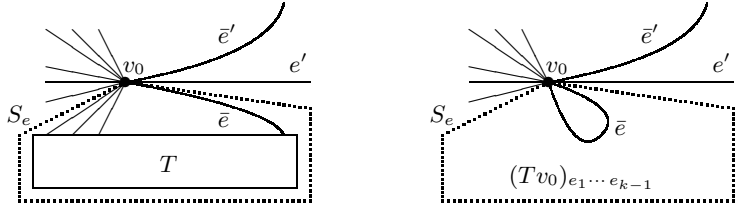

\begin{figure}[htb]
\centering
\myfbox{
\begin{picture}(110,75)(-10,-15)
\small
\put(30,30){\circle*{4}} \put(90,30){\circle*{4}}
\put(29,34){$v_0$}  \put(59,34){$e$} \put(95,34){$e^\prime$}
\put(80,15){$\bar e$} \put(80,45){$\bar e^\prime$}
\put(30,30){\line(-1, 2){10}}
\put(30,30){\line(-1, 1){20}}
\put(30,30){\line(-3, 2){30}}
\put(30,30){\line(-5, 1){30}}
\put(30,30){\line(-1, 0){30}}
\put(30,30){\line(-4,-1){30}}
\put(30,30){\line(-3,-2){30}}
\put(30,30){\line(-1,-1){20}}
\put(30,30){\line(-1,-2){10}}
\put(30,30){\line( 1, 0){60}}
\put(90,30){\line(0, 1){30}}
\put(90,30){\line(1, 0){10}}
\put(90,30){\line(0,-1){20}}
\put(-5,10){\line(1,0){100}} \put(95,10){\line(0,-1){20}} \put(-5,-10){\line(1,0){100}} \put(-5,10){\line(0,-1){20}}
\put(43,-4){$T$}
\put(-15,15){$S$}
\thicklines
\qbezier[20](30,30)(10,20)(-10,10)
\qbezier[15](-10,10)(-10,-2.5)(-10,-15)
\qbezier[50](-10,-15)(45,-15)(100,-15)
\qbezier[18](100,-15)(100,2.5)(100,20)
\qbezier[20](30,30)(60,30)(90,30)
\qbezier[7](90,30)(95,25)(100,20)
\end{picture}
}
\qquad\qquad
\myfbox{
\begin{picture}(110,75)(-10,-15)
\small
\put(30,30){\circle*{4}} \put(90,30){\circle*{4}}
\put(29,34){$v_0$}  \put(59,34){$e$} \put(95,34){$e^\prime$}
\put(80,15){$\bar e$} \put(80,45){$\bar e^\prime$}
\put(30,30){\line(-1, 2){10}}
\put(30,30){\line(-1, 1){20}}
\put(30,30){\line(-3, 2){30}}
\put(30,30){\line(-5, 1){30}}
\put(30,30){\line(-1, 0){30}}
\put(30,30){\line(-4,-1){30}}
\put(30,30){\line( 1, 0){60}}
\put(90,30){\line(0, 1){30}}
\put(90,30){\line(1, 0){10}}
\qbezier(90,30)(90,05)(80,05) \qbezier(80,05)(60,05)(30,30)
\put(-15,15){$S$}
\thicklines
\qbezier[20](30,30)(10,20)(-10,10)
\qbezier[15](-10,10)(-10,-2.5)(-10,-15)
\qbezier[50](-10,-15)(45,-15)(100,-15)
\qbezier[18](100,-15)(100,2.5)(100,20)
\qbezier[20](30,30)(60,30)(90,30)
\qbezier[7](90,30)(95,25)(100,20)
\end{picture}
}
\caption{$D$ and $D_{e_1\cdots\,e_{k-1}}$}
\label{fig:D and Dprime}
\end{figure}
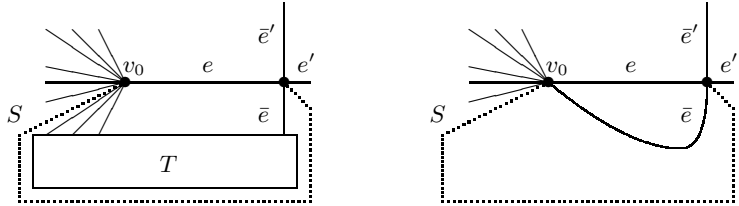

\begin{figure}[htb]
\centering
\myfbox{
\begin{picture}(110,75)(-10,-15)
\small
\put(30,30){\circle*{4}}
\put(29,34){$v_0$} \put(92,34){$e^\prime$}
\put(65,45){$\bar e^\prime$}
\put(30,30){\line(-1, 2){10}}
\put(30,30){\line(-1, 1){20}}
\put(30,30){\line(-3, 2){30}}
\put(30,30){\line(-5, 1){30}}
\put(30,30){\line(-1, 0){30}}
\put(30,30){\line(-4,-1){30}}
\put(30,30){\line( 1, 0){60}}
\qbezier(30,30)(85,40)(90,60)
\put(90,30){\line(1, 0){10}}
\put(-15,15){$S_e$}
\thicklines
\qbezier[20](30,30)(10,20)(-10,10)
\qbezier[15](-10,10)(-10,-2.5)(-10,-15)
\qbezier[50](-10,-15)(45,-15)(100,-15)
\qbezier[18](100,-15)(100,2.5)(100,20)
\qbezier[35](100,20)(65,25)(30,30)
\end{picture}
}
\caption{$D_{e_1\cdots\,e_{k}}$}
\label{fig:(Dprime)_e}
\end{figure}
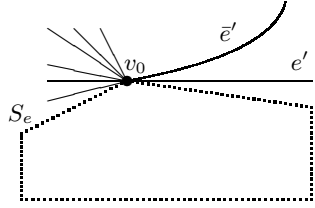

\subsection*{Case 1} As shown in Figure~\ref{fig:case1}, we consider the tangle $T$ inside $S_e$ whose end arcs are joined to $v_0$ to make a knot-spoke diagram $Tv_0$. Then $Tv_0$ is cut-point free. By applying Lemma~\ref{lem:adapt-key-BP2000} repeatedly to $Tv_0$, there is a sequence of edges $e_1,\ldots,e_{k-1}$ of $D$ inside $S$ (hence of $Tv_0$), all distinct from $\bar e$, such that the knot-spoke diagram $(Tv_0)_{e_1\cdots e_{k-1}}$, created inside $S_e$, has only one non-spoke edge which is the loop corresponding to $\bar e$.
As shown in Figure~\ref{fig:D and Dprime}, in the knot-spoke diagram $D_{e_1 \cdots\,e_{k-1}}$, the two edges $e$ and $\bar e$ together bound a region which may contain spokes. Taking $e_k=e$, the knot-spoke diagram $D_{e_1\cdots\,e_{k}}$ obtained by contracting $e$ and replacing $\bar e$ by a spoke is identical with $D_e$ outside $S_e$ and has only spokes inside $S_e$. 
Therefore condition~(3) is satisfied in this case.

\begin{figure}[htb]
\centering
\myfbox{
\begin{picture}(110,75)(-10,-15)
\small
\put(30,30){\circle*{4}}
\put(29,34){$v_0$} \put(82,24){$e^\prime$}
\put(65,13){$\bar e$} \put(65,45){$\bar e^\prime$}
\put(30,30){\line(-1, 2){10}}
\put(30,30){\line(-1, 1){20}}
\put(30,30){\line(-3, 2){30}}
\put(30,30){\line(-5, 1){31.25}}
\put(30,30){\line(-1, 0){25}}
\put(30,30){\line(-4,-1){20}}
\put(30,30){\line(-3,-2){15}}
\put(30,30){\line(-1,-1){12.5}}
\put(30,30){\line(-1,-2){8.3333333}}
\qbezier(30,30)(85,40)(90,60)
\qbezier(30,30)(70,30)(80,10)
\qbezier(30,30)(85,30)(90,10)
\put(25,10){\line(1,0){70}} \put(95,10){\line(0,-1){20}} \put(-5,-10){\line(1,0){100}} \put(-5,40){\line(0,-1){50}}
\put(-5,40){\line(1,-1){30}}
\put(43,-4){$T$}
\put(90,44){$S_e$}
\thicklines
\qbezier[20](30,30)(10,40)(-10,50)
\qbezier[35](-10,50)(-10,17.5)(-10,-15)
\qbezier[50](-10,-15)(45,-15)(100,-15)
\qbezier[30](100,-15)(100,12.5)(100,40)
\qbezier[35](100,40)(65,35)(30,30)
\end{picture}
}
\qquad\qquad
\myfbox{
\begin{picture}(110,75)(-10,-15)
\small
\put(30,30){\circle*{4}}
\put(29,34){$v_0$} \put(82,24){$\bar ee^\prime$}
\put(30,30){\line(-5, 1){31.25}}
\put(30,30){\line(-1, 0){25}}
\put(30,30){\line(-4,-1){20}}
\put(30,30){\line(-3,-2){15}}
\put(30,30){\line(-1,-1){12.5}}
\put(30,30){\line(-1,-2){8.3333333}}
\qbezier(85,20)(80,20)(80,10)
\qbezier(85,20)(90,20)(90,10)
\put(25,10){\line(1,0){70}} \put(95,10){\line(0,-1){20}} \put(-5,-10){\line(1,0){100}} \put(-5,40){\line(0,-1){50}}
\put(-5,40){\line(1,-1){30}}
\put(43,-4){$T$}
\end{picture}
}
\caption{Case 2 : $D_e$ and $\bar Tv_0$}\label{fig:case2}
\end{figure}
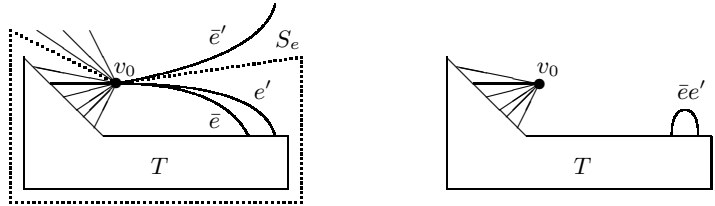

\begin{figure}[htb]
\centering
\myfbox{
\begin{picture}(170,55)(-10,-15)
\small
\put(30,30){\circle*{4}}
\put(29,34){$v_0$} \put(142,24){$\bar ee^\prime$}
\put(30,30){\line(-5, 1){31.25}}
\put(30,30){\line(-1, 0){25}}
\put(30,30){\line(-4,-1){20}}
\put(30,30){\line(-3,-2){15}}
\put(30,30){\line(-1,-1){12.5}}
\put(30,30){\line(-1,-2){8.3333333}}
\qbezier(145,20)(140,20)(140,10)
\qbezier(145,20)(150,20)(150,10)
\put(65,10){\line(1,0){10}} \put(85,10){\line(1,0){10}} \put(135,10){\line(1,0){20}}
\put(25,10){\line(0,-1){20}} \put(65,10){\line(0,-1){20}} \put(95,10){\line(0,-1){20}} \put(135,10){\line(0,-1){20}} \put(155,10){\line(0,-1){20}}
\put(-5,-10){\line(1,0){30}}
\put(65,-10){\line(1,0){10}} \put(85,-10){\line(1,0){10}} \put(135,-10){\line(1,0){20}}
\put(-5,40){\line(0,-1){50}}
\put(-5,40){\line(1,-1){30}}
\put(74,-2){$\cdots$} 
\put(25,-8){\line(5,2){40}}\put(25,8){\line(5,-2){40}}\put(45,0){\circle*{4}}\put(40,6){$w_1$}
\put(95,-8){\line(5,2){40}}\put(95,8){\line(5,-2){40}}\put(115,0){\circle*{4}}\put(108,6){$w_{m}$}
\put(145,44){$S_e$}
\end{picture}
}
\caption{$\bar Tv_0$ may have cut-points $w_1,\ldots,w_{m}$.}\label{fig:cut-points w_i}
\end{figure}
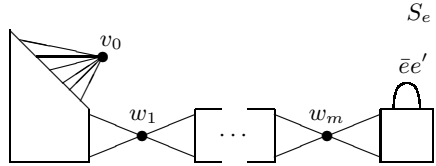

\begin{figure}[htb]
\centering
\myfbox{
\begin{picture}(130,55)(-10,-15)
\small
\put(30,30){\circle*{4}}
\put(29,34){$v_0$} 
\put(30,30){\line(-5, 1){31.25}}
\put(30,30){\line(-1, 0){25}}
\put(30,30){\line(-4,-1){20}}
\put(30,30){\line(-3,-2){15}}
\put(30,30){\line(-1,-1){12.5}}
\put(30,30){\line(-1,-2){8.3333333}}
\put(40,10){\line(1,0){ 5}} \put(60,10){\line(1,0){ 5}} \put(80,10){\line(1,0){15}} 
\put(25,10){\line(0,-1){20}} \put(40,10){\line(0,-1){20}} \put(65,10){\line(0,-1){20}} \put(80,10){\line(0,-1){20}} \put(95,10){\line(0,-1){20}} 
\put(-5,-10){\line(1,0){30}}
\put(40,-10){\line(1,0){ 5}} \put(60,-10){\line(1,0){ 5}} \put(80,-10){\line(1,0){15}} 
\put(-5,40){\line(0,-1){50}}
\put(-5,40){\line(1,-1){30}}
\put(46,-2){$\cdots$} 
\put(25,-7.5){\line(1,1){15}}\put(25,7.5){\line(1,-1){15}}\put(32.5,0){\circle*{4}}\put(27,9){$w_1$}
\put(65,-7.5){\line(1,1){15}}\put(65,7.5){\line(1,-1){15}}\put(72.5,0){\circle*{4}}\put(62,13){$w_{i-1}$}
\qbezier(95,-7.5)(102.5,0)(102.5,0) \qbezier(95,7.5)(102.5,0)(102.5,0) \put(102.5,0){\circle*{4}}
\put(99, 7){$a_i$}\put(101,-12){$b_i$}\put(108,-2){$w_i$}
\end{picture}
}
\qquad
\myfbox{
\begin{picture}(130,55)(-10,-15)
\small
\put(30,30){\circle*{4}}
\put(29,34){$v_0$} 
\put(30,30){\line(-5, 1){31.25}}
\put(30,30){\line(-1, 0){25}}
\put(30,30){\line(-4,-1){20}}
\put(30,30){\line(-3,-2){15}}
\put(30,30){\line(-1,-1){12.5}}
\put(30,30){\line(-1,-2){8.3333333}}
\put(40,10){\line(1,0){ 5}} \put(60,10){\line(1,0){ 5}} \put(80,10){\line(1,0){15}} 
\put(25,10){\line(0,-1){20}} \put(40,10){\line(0,-1){20}} \put(65,10){\line(0,-1){20}} \put(80,10){\line(0,-1){20}} \put(95,10){\line(0,-1){20}} 
\put(-5,-10){\line(1,0){30}}
\put(40,-10){\line(1,0){ 5}} \put(60,-10){\line(1,0){ 5}} \put(80,-10){\line(1,0){15}} 
\put(-5,40){\line(0,-1){50}}
\put(-5,40){\line(1,-1){30}}
\put(80,30){$E_i$}
\put(46,-2){$\cdots$} 
\put(25,-7.5){\line(1,1){15}}\put(25,7.5){\line(1,-1){15}}\put(32.5,0){\circle*{4}}\put(27,9){$w_1$}
\put(65,-7.5){\line(1,1){15}}\put(65,7.5){\line(1,-1){15}}\put(72.5,0){\circle*{4}}\put(62,13){$w_{i-1}$}
\qbezier(95,-8)(115,0)(95,8) \put(110,-2){$c_{i}$}
\end{picture}
}
\caption{$E_i$ is obtained by truncating $\bar Tv_0$ at $w_i$.}
\end{figure}
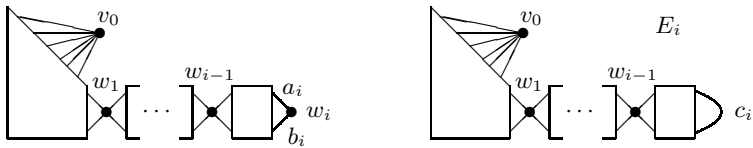

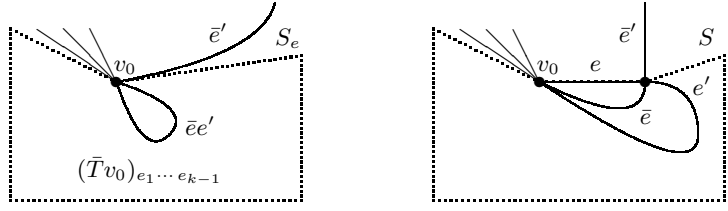
\begin{figure}[htb]
\centering
\myfbox{
\begin{picture}(110,75)(-10,-15)
\small
\put(30,30){\circle*{4}}
\qbezier(30,30)(60,20)(50,10) \qbezier(30,30)(40,0)(50,10)
\put(29,34){$v_0$} \put(56, 10){$\bar ee^\prime$}
\put(65,45){$\bar e^\prime$}
\put(30,30){\line(-1, 2){10}}
\put(30,30){\line(-1, 1){20}}
\put(30,30){\line(-3, 2){30}}
\qbezier(30,30)(85,40)(90,60)
\put(15,-6){$(\bar Tv_0)_{e_1\cdots\,e_{k-1}}$}
\put(90,44){$S_e$}
\thicklines
\qbezier[20](30,30)(10,40)(-10,50)
\qbezier[35](-10,50)(-10,17.5)(-10,-15)
\qbezier[50](-10,-15)(45,-15)(100,-15)
\qbezier[30](100,-15)(100,12.5)(100,40)
\qbezier[35](100,40)(65,35)(30,30)
\end{picture}
}
\qquad\qquad
\myfbox{
\begin{picture}(110,75)(-10,-15)
\small
\put(30,30){\circle*{4}} \put(70,30){\circle*{4}}
\put(29,34){$v_0$}  \put(49,34){$e$} \put(88,25){$e^\prime$}
\put(68,14){$\bar e$} \put(60,45){$\bar e^\prime$}
\put(30,30){\line(-1, 2){10}}
\put(30,30){\line(-1, 1){20}}
\put(30,30){\line(-3, 2){30}}
\put(30,30){\line( 1, 0){40}}
%
%
\put(70,30){\line(0, 1){30}}
\qbezier(70,30)(70,10)(30,30)\qbezier(70,30)(90,30)(90,10)\qbezier(90,10)(90,-10)(30,30)
\put(90,44){$S$}
\thicklines
\qbezier[20](30,30)(10,40)(-10,50)
\qbezier[35](-10,50)(-10,17.5)(-10,-15)
\qbezier[50](-10,-15)(45,-15)(100,-15)
\qbezier[30](100,-15)(100,12.5)(100,40)
\qbezier[15](100,40)(85,35)(70,30)
\qbezier[15](70,30)(50,30)(30,30)
\end{picture}
}
\caption{$(\bar Tv_0)_{e_1\cdots\,e_{k-1}}$ and $D_{e_1\cdots\,e_{k-1}}$}
\end{figure}

\begin{figure}[htb]
\centering
\myfbox{
\begin{picture}(110,75)(-10,-15)
\small
\put(30,30){\circle*{4}}
\put(29,34){$v_0$} 
\put(65,45){$\bar e^\prime$}
\put(30,30){\line(-1, 2){10}}
\put(30,30){\line(-1, 1){20}}
\put(30,30){\line(-3, 2){30}}
\qbezier(30,30)(85,40)(90,60)
\put(90,44){$S_e$}
\thicklines
\qbezier[20](30,30)(10,40)(-10,50)
\qbezier[35](-10,50)(-10,17.5)(-10,-15)
\qbezier[50](-10,-15)(45,-15)(100,-15)
\qbezier[30](100,-15)(100,12.5)(100,40)
\qbezier[35](100,40)(65,35)(30,30)
\end{picture}
}
\caption{$D_{e_1\cdots\,e_{k}}$}
\end{figure}
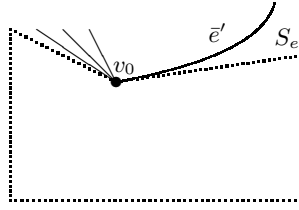

\subsection*{Case 2} As in Figure~\ref{fig:case2}, in the knot-spoke diagram $Tv_0$, the two edges $\bar e$ and $e^\prime$ are adjacently incident to $v_0$. We take off $\bar e$ and $e^\prime$ from $v_0$ and create a new edge $\bar ee^\prime$. Then we have a knot-spoke diagram $\bar Tv_0$ inside $S_e$ which may have cut-points $w_1,\ldots,w_m$ as shown by Figure~\ref{fig:cut-points w_i}.
For $i=1,\ldots,m$, let $E_i$ be the knot-spoke diagram obtained by truncating $\bar Tv_0$ at the cut-point $w_i$ and amalgamating two edges $a_i$ and $b_i$ into $c_i$.
By applying Lemma~\ref{lem:adapt-key-BP2000} to $E_1$, we obtain a sequence $e_1,\ldots,e_{k_1-1}$ of edges of $E_1$, all distinct from $c_1$, such that $(E_1)_{e_1\cdots\,e_{k_1-1}}$ has only one non-spoke edge which is a loop corresponding to $c_1$.
Taking $e_{k_1}=a_1$, the knot-spoke diagram $(E_2)_{e_1\cdots\,e_{k_1}}$ is obtained from $(E_2)_{e_1\cdots\,e_{k_1-1}}$ by contracting $a_1$ and replacing the loop $b_1$ by a spoke. Proceeding in this manner, we obtain a sequence $e_{k_i+1},\ldots,e_{k_{(i+1)}-1},e_{k_{(i+1)}}=a_{i+1}$ of edges of $E_{i+1}\setminus E_{i}$ such that $(E_{i+1})_{e_1\cdots\,e_{k_{(i+1)}-1}}$ has only one non-spoke edge which is a loop corresponding to $c_{i+1}$, for $i=1,\ldots,m-1$.
Taking $e_{k_m}=\bar e$, the knot-spoke diagram $(\bar Tv_0)_{e_1\cdots\,e_{k_m}}$ is cut-point free. Applying Lemma~\ref{lem:adapt-key-BP2000}, we obtain a sequence $e_{k_m+1},\ldots,e_{k-1}$ of edges of $\bar Tv_0\setminus E_m$ such that the knot-spoke diagram $(\bar Tv_0)_{e_1\cdots\,e_{k-1}}$ has only one non-spoke edge corresponding to the loop $\bar ee^\prime$.
Finally, by taking $e_k=\bar e$, the knot-spoke diagram $D_{e_1\cdots\,e_{k}}$ obtained by contracting $\bar e$ and replacing the loops corresponding to $e$ and $e^\prime$ by spokes, has no non-spoke edges inside $S_e$ and no non-spoke edges outside $S_e$ has been contracted or became a spoke except $e$. Therefore condition~(3) is satisfied in this case.

\subsection*{Case 3}
Since $S_e$ and $S^\prime_e$ meet in $v_0$, they divide the plane into two bounded regions and one unbounded region. We may assume that $S^\prime_e$ is the boundary of the unbounded region. Then we have two cut-point free knot-spoke diagrams $Tv_0$ and $T^\prime v_0$ inside $S_e$ and between $S_e$ and $S^\prime_e$, respectively, as in Figure~\ref{fig:case3}.
By applying Lemma~\ref{lem:adapt-key-BP2000} to $Tv_0$, we obtain a sequence $e_1,\ldots,e_{j}$ of edges of $Tv_0$ such that $(Tv_0)_{e_1\cdots\,e_{j}}$ has only one non-spoke edge which is a loop corresponding to $\bar e$.
By applying Lemma~\ref{lem:adapt-key-BP2000} to $T^\prime v_0$, we obtain a sequence $e_{j+1},\ldots,e_{k-1}$ of edges of $T^\prime v_0$ such that $(T^\prime v_0)_{e_{j+1}\cdots\,e_{k-1}}$ has only one non-spoke edge which is a loop corresponding to $e^\prime$.
Then
$$(Tv_0)_{e_1\cdots\,e_{j}}\cup(T^\prime v_0)_{e_{j+1}\cdots\,e_{k-1}}=(Tv_0\cup T^\prime v_0)_{e_{1}\cdots\,e_{k-1}}.$$
Finally, by taking $e_k=\bar e$, the knot-spoke diagram $D_{e_1\cdots\,e_{k}}$ obtained by contracting $\bar e$ and replacing the loops corresponding to $e$ and $e^\prime$ by spokes, has no non-spoke edges inside $S^\prime_e$ and no non-spoke edges outside $S^\prime_e$ has been contracted or became a spoke except $e$. Therefore condition~(3) is satisfied in this case.
\end{proof}

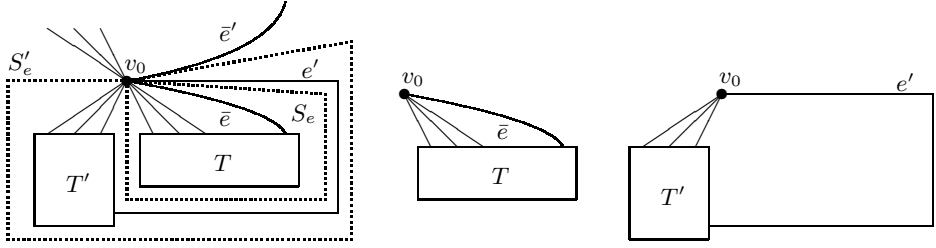
\begin{figure}[t]
\centering
\myfbox{
\begin{picture}(130,90)(-15,-30)
\small
\put(30,30){\circle*{4}}
\put(29,34){$v_0$} \put(96,31){$e^\prime$}
\put(65,13){$\bar e$} \put(65,45){$\bar e^\prime$}
\put(30,30){\line(-1, 2){10}}
\put(30,30){\line(-1, 1){20}}
\put(30,30){\line(-3, 2){30}}
\put(30,30){\line(-3,-2){30}}
\put(30,30){\line(-1,-1){20}}
\put(30,30){\line(-1,-2){10}}
\put(30,30){\line( 3,-2){30}}
\put(30,30){\line( 1,-1){20}}
\put(30,30){\line( 1,-2){10}}
\put(30,30){\line( 1, 0){80}} \put(110,30){\line( 0,-1){50}} \put(110,-20){\line(-1, 0){85}}
\qbezier(30,30)(85,40)(90,60)
\qbezier(30,30)(85,20)(90,10)
\put(35,10){\line(1,0){60}} \put(95,10){\line(0,-1){20}} \put(35,-10){\line(1,0){60}} \put(35,10){\line(0,-1){20}}
\put(63,-4){$T$}
\put(-5,10){\line(1,0){30}} \put(25,10){\line(0,-1){35}} \put(-5,-25){\line(1,0){30}} \put(-5,10){\line(0,-1){35}}
\put(7,-12){$T^\prime$}
\put(92,15){$S_e$}
\thicklines
\qbezier[22](30,30)(30,7.5)(30,-15)
\qbezier[40](30,-15)(67.5,-15)(105,-15)
\qbezier[20](105,-15)(105,5)(105,25)
\qbezier[42](105,25)(67.5,27.5)(30,30)
\put(-15,34){$S^\prime_e$}
\qbezier[22](30,30)(7.5,30)(-15,30)
\qbezier[30](-15,30)(-15,0)(-15,-30)
\qbezier[65](-15,-30)(50,-30)(115,-30)
\qbezier[35](115,-30)(115,7.5)(115,45)
\qbezier[42](115,45)(72.5,37.5)(30,30)
\end{picture}
}
\quad
\myfbox{
\begin{picture}(65,65)(30,-25)
\small
\put(30,30){\circle*{4}}
\put(29,34){$v_0$} 
\put(65,13){$\bar e$} 
\put(30,30){\line( 3,-2){30}}
\put(30,30){\line( 1,-1){20}}
\put(30,30){\line( 1,-2){10}}
\qbezier(30,30)(85,20)(90,10)
\put(35,10){\line(1,0){60}} \put(95,10){\line(0,-1){20}} \put(35,-10){\line(1,0){60}} \put(35,10){\line(0,-1){20}}
\put(63,-4){$T$}
\end{picture}
}
\quad 
\myfbox{
\begin{picture}(115,65)(-5,-25)
\small
\put(30,30){\circle*{4}}
\put(29,34){$v_0$} \put(96,31){$e^\prime$}
\put(30,30){\line(-3,-2){30}}
\put(30,30){\line(-1,-1){20}}
\put(30,30){\line(-1,-2){10}}
\put(30,30){\line( 1, 0){80}} \put(110,30){\line( 0,-1){50}} \put(110,-20){\line(-1, 0){85}}
\put(-5,10){\line(1,0){30}} \put(25,10){\line(0,-1){35}} \put(-5,-25){\line(1,0){30}} \put(-5,10){\line(0,-1){35}}
\put(7,-12){$T^\prime$}
\end{picture}
}
\caption{Case 3: $D_e$, $Tv_0$ and $T^\prime v_0$}\label{fig:case3}
\end{figure}

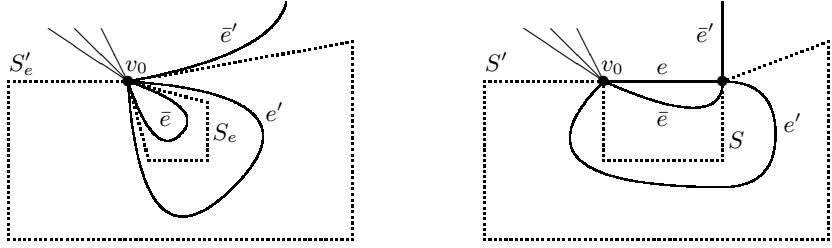
\begin{figure}[htb]
\centering
\myfbox{
\begin{picture}(130,90)(-15,-30)
\small
\put(30,30){\circle*{4}}
\put(29,34){$v_0$} \put(82,15){$e^\prime$}
\put(42,13){$\bar e$} \put(65,45){$\bar e^\prime$}
\put(30,30){\line(-1, 2){10}}
\put(30,30){\line(-1, 1){20}}
\put(30,30){\line(-3, 2){30}}
\qbezier(30,30)(85,40)(90,60)
\put(62,8){$S_e$}
\qbezier(30,30)(60,20)(50,10) \qbezier(30,30)(40,0)(50,10)
\put(-15,34){$S^\prime_e$}
\qbezier(30,30)(35,-45)(70,-10) \qbezier(30,30)(105,25)(70,-10)
\thicklines
\qbezier[15](30,30)(34,15)(38,0)
\qbezier[11](38,0)(49,0)(60,0)
\qbezier[11](60,0)(60,11)(60,22)
\qbezier[15](60,22)(45,26)(30,30)
\qbezier[22](30,30)(7.5,30)(-15,30)
\qbezier[30](-15,30)(-15,0)(-15,-30)
\qbezier[65](-15,-30)(50,-30)(115,-30)
\qbezier[35](115,-30)(115,7.5)(115,45)
\qbezier[42](115,45)(72.5,37.5)(30,30)
\end{picture}
}
\qquad\qquad
\myfbox{
\begin{picture}(130,90)(-15,-30)
\small
\put(30,30){\circle*{4}}\put(75,30){\circle*{4}}
\put(29,34){$v_0$} \put(98,11){$e^\prime$}
\put(50,33){$e$} \put(50,13){$\bar e$} \put(65,45){$\bar e^\prime$}
\put(30,30){\line(-1, 2){10}}
\put(30,30){\line(-1, 1){20}}
\put(30,30){\line(-3, 2){30}}
\put(30,30){\line( 1, 0){45}}
\put(75,30){\line(0,1){30}}
\qbezier(75,30)(75,10)(30,30)
\qbezier(75,30)(95,30)(95,10)\qbezier(95,10)(95,-10)(75,-10)\qbezier(30,30)(-10,-10)(75,-10)
\put(77,5){$S$}
\put(-15,34){$S^\prime$}
\thicklines
\qbezier[15](30,30)(30,15)(30,0)
\qbezier[22](30, 0)(52.5, 0)(75, 0)
\qbezier[15](75,0)(75,15)(75,30)
\qbezier[22](30,30)(7.5,30)(-15,30)
\qbezier[30](-15,30)(-15,0)(-15,-30)
\qbezier[65](-15,-30)(50,-30)(115,-30)
\qbezier[35](115,-30)(115,7.5)(115,45)
\qbezier[20](115,45)(95,37.5)(75,30)
\qbezier[15](75,30)(52.5,30)(30,30)
\end{picture}
}
\caption{$(Tv_0\cup T^\prime v_0)_{e_1\cdots\,e_{k-1}}$ and $D_{e_1\cdots\,e_{k-1}}$}\label{fig:??}
\end{figure}

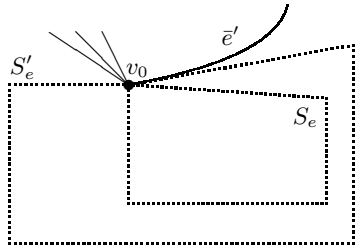
\begin{figure}[htb]
\centering
\myfbox{
\begin{picture}(130,90)(-15,-30)
\small
\put(30,30){\circle*{4}}
\put(29,34){$v_0$}
\put(65,45){$\bar e^\prime$}
\put(30,30){\line(-1, 2){10}}
\put(30,30){\line(-1, 1){20}}
\put(30,30){\line(-3, 2){30}}
\qbezier(30,30)(85,40)(90,60)
\put(92,15){$S_e$}
\thicklines
\qbezier[22](30,30)(30,7.5)(30,-15)
\qbezier[40](30,-15)(67.5,-15)(105,-15)
\qbezier[20](105,-15)(105,5)(105,25)
\qbezier[42](105,25)(67.5,27.5)(30,30)
\put(-15,34){$S^\prime_e$}
\qbezier[22](30,30)(7.5,30)(-15,30)
\qbezier[30](-15,30)(-15,0)(-15,-30)
\qbezier[65](-15,-30)(50,-30)(115,-30)
\qbezier[35](115,-30)(115,7.5)(115,45)
\qbezier[42](115,45)(72.5,37.5)(30,30)
\end{picture}
}
\caption{$D_{e_1\cdots\,e_{k}}$}\label{fig:last in proof}
\end{figure}

Suppose that $D$ is a prime non-alternating minimal crossing link diagram. An innermost or outermost region of the plane bounded by edges of $D$ is called an \emph{alternating region\/} of $D$ if its boundary edges are all alternating and a \emph{non-alternating region\/} of $D$ otherwise. Let $D^\star$ denote the subgraph of the dual graph of $D$ whose edges are the duals of non-alternating edges of $D$. One should notice that among the non-alternating edges, the undercrossing ones and the overcrossing ones appear alternatingly around any non-alternating region. This shows that, in $D^\star$, every vertex is even-valent, every edge belongs to a cycle, and every cycle has an even number of edges. Because $D$ is prime, no bigons exist in $D^\star$.
Now we claim that at least one of the three following cases occur.
\begin{itemize}
\item{} \textbf{\it Case I.} There is an arc of $D$ which crosses over (or under) three times consecutively.
\item{} \textbf{\it Case II.} The two edges of an arc crossing one end of a non-alternating edge are both alternating.
\item{} \textbf{\it Case III.} One of the two edges of an arc crossing one end of a non-alternating edge is alternating and the other is non-alternating.
\end{itemize}

Suppose that an innermost cycle $C$ in $D^\star$ does not contain any alternating region of $D$. Then the part of $D$ inside $C$ is a tree $T$ consisting of edges and half edges. There is a vertex $v$ in $T$ where at least three half edges are incident. The extension of the two nonadjacent half edges is an arc which has three consecutive overcrossings or undercrossings. Suppose that an innermost cycle $C$ in $D^\star$ contains an alternating region, then, because non-alternating regions cannot be isolated, there exists a vertex at which at least one alternating region and at least two non-alternating regions meet. This proves our claim.

\begin{figure}[t]
\centering
\small
\myfbox{
\begin{picture}(200,110)(-100,-60)
{\thicklines\qbezier(-70,40)(0,20)(70,40)\put(-30,23){$\alpha$}}
\put(-98, 2){$v_0$}\put(-91.3,-5.4){\circle*{4}}
\put(-73,45){$v_1$}\put(-60,37.3){\circle*{4}}
\put(-12,37){$v_2$}\put(  0,30  ){\circle*{4}}
\put( 63,45){$v_3$}\put( 60,37.3){\circle*{4}}
\put(-40,-5){$R_1$} \put(30,-5){$R_2$}
\put(-32,37){$g_1$} \put(25,37){$g_2$}
\put(-60,50){\line(0,-1){8}}\qbezier(-60,33)(-60,10)(-100,-10)\put(-80,12){$f_1$}
\qbezier(-100,0)(-80,-10)(-80,-35)\put(-84,-12){$e_1$}
\qbezier[18](-20,-41)(-35,-30)(-55,-40)
\qbezier(-90,-30)(-65,-20)(-50,-40)\put(-65,-25){$e_2$}
\put(  0,45){\line(0,-1){10}} \put(  0,25){\line(0,-1){70}}\put(-13, 0){$f_2$}\put(-4,-60){$D$}
\qbezier(-30,-40)(0,-30)(30,-40)\put(-20,-32){$e_r$}\put(3,-32){$e_{r+1}$}
\qbezier( 20,-41)( 35,-30)( 55,-40)\put(30,-32){$e_{r+2}$}
\qbezier[20]( 87,-30)( 62,-20)( 47,-40)
\qbezier( 100,0)( 80,-10)( 80,-35)\put(72,-17){$e_{s}$}
\put( 60,50){\line(0,-1){8}}\qbezier( 60,33)( 60,10)(100,-10)\put(75,12){$f_3$}
\end{picture}
}
\myfbox{
\begin{picture}(140,110)(-70,-60)
{\thicklines\qbezier(-70,40)(0,20)(70,40)\put(-30,23){$\alpha$}}
\put(  2,-47){$v_0$}\put(0,-36){\circle*{4}}
\put(-73,45){$v_1$}\put(-60,37.3){\circle*{4}}
\put(-12,37){$v_2$}\put(  0,30  ){\circle*{4}}
\put( 63,45){$v_3$}\put( 60,37.3){\circle*{4}}
\put(-30, 5){$R_1$} \put(20, 5){$R_2$}
\put(-32,37){$g_1$} \put(25,37){$g_2$}
\put(-60,50){\line(0,-1){8}}\qbezier(-60,33)(-60,10)( 5,-40)\put(-66,12){$f_1$}
\put(  0,45){\line(0,-1){10}} \put(  0,25){\line(0,-1){70}}\put(-11,-10){$f_2$}
\put(-10,-36){\line(1,0){20}}\put(-8,-40){\line(2,1){16}}\put( 8,-40){\line(-2, 1){16}}
\put( 60,50){\line(0,-1){8}}\qbezier( 60,33)( 60,10)(-5,-40)\put(58,12){$f_3$}\put(-4,-60){$D^\prime$}
\end{picture}
}
\caption{Case I :  $D$ and $D^\prime$}\label{fig:caseI}
\end{figure}
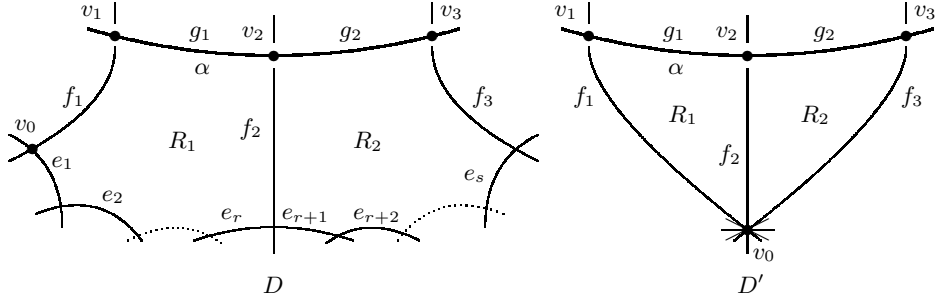

\subsection*{Case I}
Suppose there is an arc $\alpha$ of $D$ consecutively crossing over (resp. under) at the vertices $v_1$, $v_2$ and $v_3$. For $i=1,2$, let $g_i$ be the edge joining $v_i$ and $v_{i+1}$. On one side of this arc there are two adjacent non-alternating regions $R_1$ and $R_2$, whose boundary edges are $f_1,e_1,\ldots,e_r,f_2,\overline{v_1v_2}$ and $f_2,e_{r+1},\ldots,e_s,f_3,\overline{v_2v_3}$, respectively, as shown in Figure~\ref{fig:caseI}. Let $v_0$ be the vertex which is the common endpoint of $f_1$ and $e_1$. By contracting the edges $e_1,\ldots,e_s$ to $v_0$ one after another, we can obtain a new knot-spoke diagram $D_{e_1\cdots\,e_s}$ in which the regions $R_1$, $R_2$ become triangular. If $v_0$ is a cut-point, then find  $i$ such that $e_i$ is the first edge that makes $v_0$ into a cut-point of $D_{e_1\cdots\,e_i}$, then there exists a simple closed curve $S_{e_i}$ meeting $D_{e_1\cdots\,e_i}$ in $v_0$ having the regions $R_1$ and $R_2$ in one side. We apply Proposition~\ref{prop:main tool} to deform the non-spoke edges in the other side of $S_{e_i}$ either by contracting or by making into spokes. We can repeat this process until we obtain a cut-point free knot-spoke diagram $D^\prime$ in which $R_1$ and $R_2$ are triangular  as in Figure~\ref{fig:caseI}. So far we haven't changes the sum of the number of regions and the number of spokes, hence
\begin{equation}
s(D^\prime)+r(D^\prime)=r(D)=c(D)+2.\notag
\end{equation}
Let $D^{\prime\prime}$ be obtained from $D^{\prime}$ by isotoping the arc $\alpha=g_1\cup g_2$ over (resp. under) the vertex $v_0$. This is the same as contracting the edges $f_1,f_2,f_3$ simultaneously and then shrinking the two loops obtained by the overcrossing (resp. undercrossing) arc $\alpha$ to a point above (resp. below) $v_0$. This is possible because $\alpha$ can be placed above (resp. below) any other edges incident to $v_0$. Then $D^{\prime\prime}$ is a knot-spoke diagram satisfying
\begin{equation}\label{eqn:s+r=c}
s(D^{\prime\prime})+r(D^{\prime\prime})=r(D)-2=c(D).\tag{$\dagger$}
\end{equation}
Unless $D^{\prime\prime}$ has a cut-point, we are done in this case.

\begin{figure}[t]
\centering
\small
\myfbox{
\begin{picture}(140,125)(-70,-35)
{\thicklines\qbezier(-70,70)(0,50)(70,70)}
\put(  2,8){$v_0$}
\put(0,-1){\circle*{4}}
\put(-50,65){\circle*{4}}
\put(  0,60  ){\circle*{4}}
\put( 50,65){\circle*{4}}
\put(60,78){$A_1$} \put(-30,73){$A_2$} \put(-55,25){$A_3$} \put(-29,-27){$A_4$} \put(1,-32){$\cdots$} \put(24,-12){$A_n$}
\put(-65,78){$B_1$} \put( 15,73){$B_2$} \put( 50,25){$B_3$} \put( 20,-27){$B_4$} \put(-11,-32){$\cdots$} \put(-35,-12){$B_n$}
\put(-20,30){$R_1$} \put(10,30){$R_2$}
\put(-50,90){\line(0,-1){21}}\qbezier(-50,61)(-50,20)(35,-20)\put(-44,44){$f_1$}
\put(  0,85){\line(0,-1){20}} \put(  0,55){\line(0,-1){90}}\put(3,44){$f_2$}
\put(-40,-1){\line(1,0){80}}\put(0,-1){\line(-1,-2){16}}\put(0,-1){\line( 1,-2){16}}
\put( 50,90){\line(0,-1){21}}\qbezier( 50,61)( 50,20)(-35,-20)\put(36,44){$f_3$}
\end{picture}
}
\caption{Regions around $R_1\cup R_2$ in $D^\prime$}\label{fig:regions}
\end{figure}
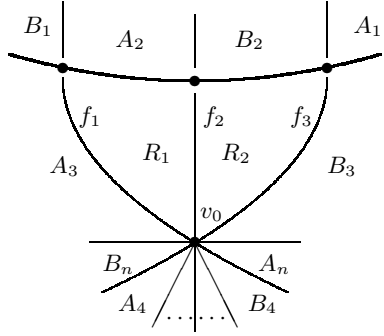

\begin{figure}[b]
\centering
\small
\myfbox{
\begin{picture}(170,110)(-70,-60)
{\thicklines\qbezier(-70,40)(0,20)(70,40)\put(-30,23){$\alpha$}}
\put(-9,-26){$v_0$}\put(0,-36){\circle*{4}}
\put(-73,45){$v_1$}\put(-60,37.3){\circle*{4}}
\put(-12,37){$v_2$}\put(  0,30  ){\circle*{4}}
\put( 63,45){$v_3$}\put( 60,37.3){\circle*{4}}
\put(-33,2){$R_1$} \put(30,-5){$R_2$}
\put(-32,37){$g_1$} \put(25,37){$g_2$}
\put(-60,50){\line(0,-1){8}}\qbezier(-60,33)(-60,10)( 5,-40)\put(-47,10){$f_1$}
%
\put(  0,45){\line(0,-1){10}} \put(  0,25){\line(0,-1){70}}\put(-13, 0){$f_2$}\put(-10,-60){$D_E$ (or $D_{E^\prime}$)}
\qbezier(-5,-38)(10,-30)(30,-40)\put(3,-32){$e_{r+1}$}
\qbezier( 20,-41)( 35,-30)( 55,-40)\put(30,-32){$e_{r+2}$}
\qbezier[20]( 87,-30)( 62,-20)( 47,-40)
\qbezier( 100,0)( 80,-10)( 80,-35)\put(72,-17){$e_{s}$}
\put( 60,50){\line(0,-1){8}}\qbezier( 60,33)( 60,10)(100,-10)\put(75,12){$f_3$}
\end{picture}
}
\myfbox{
\begin{picture}(170,110)(-70,-60)
{\thicklines\qbezier(-70,40)(-55,-30)(0,-36)\qbezier(0,-36)(30,30)(70,41)}
\put(-9,-26){$v_0$}\put(0,-36){\circle*{4}}
\put( 63,45){$v_3$}\put( 60,37.3){\circle*{4}}
%
\put(40,-5){$R_2$} \put(21,17){$g_2$}
\qbezier(-60,33)(-60,10)( 5,-40)
%
\put(  0,25){\line(0,-1){70}}
\put(14,-60){$D_1$}
\qbezier(-5,-38)(10,-30)(30,-40)\put(6,-32){$e_{r+1}$}
\qbezier( 20,-41)( 35,-30)( 55,-40)\put(30,-32){$e_{r+2}$}
\qbezier[20]( 87,-30)( 62,-20)( 47,-40)
\qbezier( 100,0)( 80,-10)( 80,-35)\put(72,-17){$e_{s}$}
\put( 60,50){\line(0,-1){8}}\qbezier( 60,33)( 60,10)(100,-10)\put(75,12){$f_3$}
\end{picture}
}
\caption{Case I (2)}\label{fig:caseI(2)}
\end{figure}
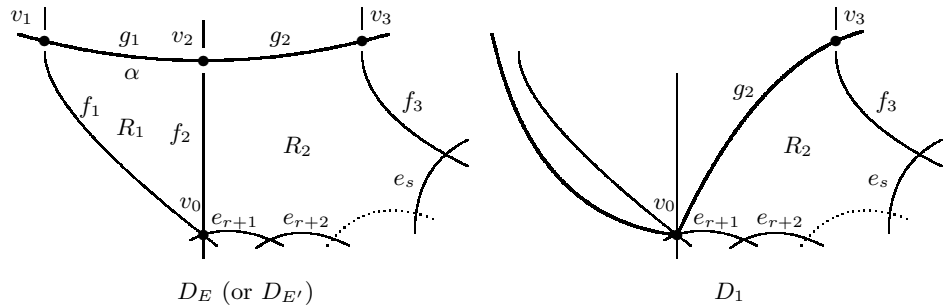

Suppose that $D^{\prime\prime}$ has a cut-point. As the cut-point must be the vertex $v_0$ and the local diagram around $R_1\cup R_2$ in $D^\prime$ is symmetric, we only need to check the cases when two regions $A_i$ and $A_j$, labeled as in Figure~\ref{fig:regions}, indicate the same region, for $i\ne j$. Because $D$ is a prime minimal crossing diagram, we do not need to consider the cases $A_i=A_{i+1}$ for $i=1,\ldots,n-1$, or $A_1=A_n$. For the case $A_1=A_n$, it may happen that $A_1$ has an edge which form a bigon region with $f_3$. This edge is replaced by a spoke when $f_3$ is contracted. Because $D^\prime$ is cut-point free, we have $A_i\ne A_j$ for $3\le i<j\le n$. Therefore the nontrivial cases occur only if $n\ge4$, and they are
\begin{enumerate}
\item $A_1=A_i$ or $B_1=B_i$ for $i=3,\ldots,n-1$,
\item $A_2=A_i$ or $B_2=B_i$ for $i=4,\ldots,n$.
\end{enumerate}
If case (1) occurs only, we can apply Proposition~\ref{prop:main tool} to $D$ outside the region $R_1\cup R_2$ to obtain a cut-point free knot-spoke diagram $D^{\prime\prime}$ satisfying the equation (\ref{eqn:s+r=c}).
If case (2) occurs, there is a simple closed curve $S$ meeting $D$ in three points which are a point on $g_1$, a point on $f_2$, and a vertex of $\partial R_2$ not incident to $f_2$ or $f_3$. Contracting the edges $e_1,\ldots,e_r$ and applying Proposition~\ref{prop:main tool} if necessary, we obtain a sequence $E$ of edges to obtain a cut-point free knot-spoke diagram $D_E$, as illustrated in Figure~\ref{fig:caseI(2)},  in which $R_1$ is triangular and $R_2$ is as in $D$. Isotoping $g_1$ over $v_0$ and replacing any simple loops created by this isotopy with spokes, we obtain a knot-spoke diagram $D_1$ satisfying
\begin{equation}\label{eqn:s+r=c+1}
s(D_1)+r(D_1)=r(D)-1=c(D)+1.\tag{$\ddagger$}
\end{equation}
If $v_0$ is a cut-point of $D_1$, then it must be case (1), and hence we can apply Proposition~\ref{prop:main tool} to obtain another sequence of edges $E^\prime$ so that the regions $R_1$ and $R_2$ in $D_{E^\prime}$ are as in $D_E$. After the same isotopy of $g_1$, we obtain a cut-point free knot-spoke diagram, again denoted by $D_1$, satisfying equation~(\ref{eqn:s+r=c+1}). Contracting the edges $e_{r+1},\ldots,e_s,f_3$, one by one, we obtain a spoke of $(D_1)_{e_{r+1}\cdots\,e_sf_3}$ created by the arc $g_2$ which can be placed above (resp. below) any other edges nearby. This spoke can be shrunk to a point above (resp. below) any other edges at $v_0$, creating a knot-spoke diagram $D^{\prime\prime}$ satisfying equation~(\ref{eqn:s+r=c}). In the case $D^{\prime\prime}$ has a cut point, we can apply Proposition~\ref{prop:main tool} to obtain another sequence $F$ of edges creating an uppermost (resp. lowermost) spoke created by $g_2$ in the cut-point free knot-spoke diagram $D^{\prime\prime}_F$. Shrinking the spoke, we obtain a cut-point free knot-spoke diagram, again denoted by $D^{\prime\prime}$ satisfying equation~(\ref{eqn:s+r=c}).

\begin{figure}[h]
\centering
\small
\myfbox{
\begin{picture}(200,95)(-100,-45)
\qbezier(-70,40.5)(-40,30)(-5,30) \qbezier( 5,30)( 40,30)( 70,40.5) 
\put(-98, 2){$v_0$}\put(-91.3,-5.4){\circle*{4}}
\put(-73,45){$v_1$}\put(-60,37.3){\circle*{4}}
\put(-12,37){$v_2$}\put(  0,30  ){\circle*{4}}
\put( 63,45){$v_3$}\put( 60,37.3){\circle*{4}}
\put(-40,-5){$R_1$} \put(30,-5){$R_2$}
\put(-32,37){$g_1$} \put(25,37){$g_2$}
\put(-60,50){\line(0,-1){8}}\qbezier(-60,33)(-60,10)(-100,-10)\put(-80,12){$f_1$}
\qbezier(-100,0)(-80,-10)(-80,-35)\put(-84,-12){$e_1$}
\qbezier[18](-20,-41)(-35,-30)(-55,-40)
\qbezier(-90,-30)(-65,-20)(-50,-40)\put(-65,-25){$e_{1}$}
{\thicklines\put(  0,45){\line(0,-1){90}}
}\put(-11,-10){$f_2$}
\qbezier(-30,-40)(-20,-35)(-5,-35) \qbezier( 5,-35)( 20,-35)( 30,-40) \put(-20,-32){$e_r$}\put(3,-32){$e_{r+1}$}
\qbezier( 20,-41)( 35,-30)( 55,-40)\put(30,-32){$e_{r+2}$}
\qbezier[20]( 87,-30)( 62,-20)( 47,-40)
\qbezier( 100,0)( 80,-10)( 80,-35)\put(72,-17){$e_{s}$}
\put( 60,50){\line(0,-1){8}}\qbezier( 60,33)( 60,10)(100,-10)\put(75,12){$f_3$}
\end{picture}
}
\myfbox{
\begin{picture}(140,95)(-70,-45)
\qbezier(-70,40.5)(-40,30)(-5,30) \qbezier( 5,30)( 40,30)( 70,40.5) 
\put(  2,-47){$v_0$}\put(0,-36){\circle*{4}}
\put(-73,45){$v_1$}\put(-60,37.3){\circle*{4}}
\put(-12,37){$v_2$}\put(  0,30  ){\circle*{4}}
\put( 63,45){$v_3$}\put( 60,37.3){\circle*{4}}
\put(-30, 5){$R_1$} \put(20, 5){$R_2$}
\put(-32,37){$g_1$} \put(25,37){$g_2$}
\put(-60,50){\line(0,-1){8}}\qbezier(-60,33)(-60,10)( 5,-40)\put(-66,12){$f_1$}
{\thicklines\put(  0,45){\line(0,-1){90}}
}\put(-11,-10){$f_2$}
\put(-10,-36){\line(1,0){20}}\put(-8,-40){\line(2,1){16}}\put( 8,-40){\line(-2, 1){16}}
\put( 60,50){\line(0,-1){8}}\qbezier( 60,33)( 60,10)(-5,-40)\put(58,12){$f_3$}
\end{picture}
}
\caption{Case II}\label{fig:caseII}
\end{figure}

\subsection*{Case II}
As shown in Figure~\ref{fig:caseII}, two non-alternating regions $R_1$, $R_2$ are adjacent along a non-alternating edge $f_2$ and the edges $g_1,g_2$ belonging to $\partial R_1,\partial R_2$, respectively, are incident to the endpoint $v_2$ of $f_2$. We may assume that there are two horizontal planes $P$ and $Q$ such that $P$ separates $f_2$ from the other edges of $\partial R_1$ and $\partial R_2$ and $Q$ separates $f_2$, $g_1$, $g_2$ from the rest. The same edge contractions and isotopies as in \textbf{\it Case I\/} can be done within the respective parts divided by $P$ and $Q$ to obtain a cut-point free knot-spoke diagram $D^{\prime\prime}$ satisfying equation~(\ref{eqn:s+r=c}).

\begin{figure}[t]
\centering\small
\myfbox{
\begin{picture}(140,110)(-70,-60)
\put(-30, 40){\line(1,0){5}} \put(-15, 40){\line(1,0){5}} \put( 10, 40){\line(1,0){5}} \put( 25, 40){\line(1,0){5}}
\put(-20,5){\line(0,1){45}} \put( 20,5){\line(0,1){45}}
\put(-70,0){\line(1,0){5}} \put(-55,0){\line(1,0){110}} \put(65,0){\line(1,0){5}}
\put(-60,-10){\line(0,1){20}} \put( 60,-10){\line(0,1){20}}
\put(-30,-40){\line(1,0){20}} \put( 10,-40){\line(1,0){20}}
\put(-20,-35){\line(0,1){30}} \put( 20,-35){\line(0,1){30}}
\put(-20,-50){\line(0,1){5}} \put( 20,-50){\line(0,1){5}}
\put(60,0){\circle*{4}} \put(62,3){$v_0$}
\put(35,-22){$R_0$}
\put(-5,-22){$R_1$} \put(-5,20){$R_2$}
\put(22,-22){$e_1$} \put(-18,-22){$e_2$}
\put(22,18){$e^\prime_1$} \put(-18,18){$e^\prime_2$}
\put(31,3){$f_0$} \put(-9,3){$f_1$} \put(-49,3){$f_2$}
\qbezier[30](30,-40)(60,-40)(60,-10) \qbezier[10](-10,-40)(0,-40)(10,-40)
\put(-5,-60){(1)}
\end{picture}
}
\qquad
\myfbox{
\begin{picture}(140,110)(-70,-60)
\put(-30, 40){\line(1,0){5}} \put(-15, 40){\line(1,0){5}} \put( 10, 40){\line(1,0){20}}
\put(-20,5){\line(0,1){45}} \put( 20,5){\line(0,1){30}} \put( 20,45){\line(0,1){5}}
\put(-70,0){\line(1,0){5}} \put(-55,0){\line(1,0){110}} \put(65,0){\line(1,0){5}}
\put(-60,-10){\line(0,1){20}} \put( 60,-10){\line(0,1){20}}
\put(-30,-40){\line(1,0){20}} \put( 10,-40){\line(1,0){5}} \put( 25,-40){\line(1,0){5}}
\put(-20,-35){\line(0,1){30}} 
\put(-20,-50){\line(0,1){5}} \put( 20,-50){\line(0,1){45}}
\put(-5,-22){$R_1$} \put(-5,20){$R_2$} \put(-9,3){$f_1$}
\put(-5,-60){(2)}
\end{picture}
}
\caption{Case III}\label{fig:caseIII}
\end{figure}
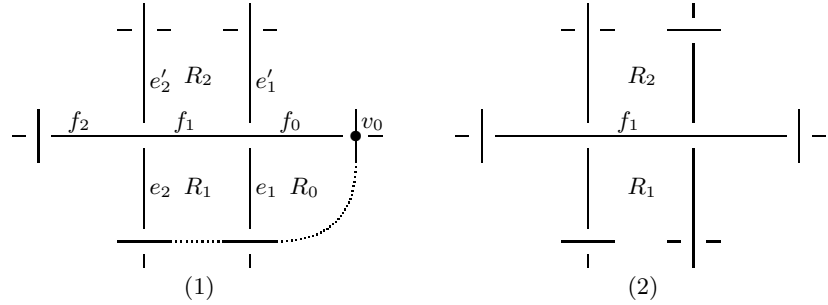

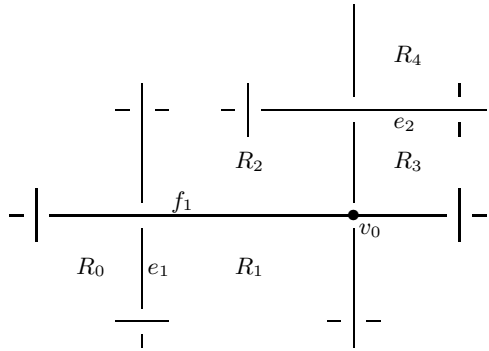
\begin{figure}[h]
\centering\small
\myfbox{
\begin{picture}(180,130)(-70,-50)
\put(-30, 40){\line(1,0){5}} \put(-15, 40){\line(1,0){5}} \put(10, 40){\line(1,0){5}} \put( 25, 40){\line(1,0){85}}
\put(20,30){\line(0,1){20}} \put(100,30){\line(0,1){5}} \put(100,45){\line(0,1){5}}
\put(-20,5){\line(0,1){45}} \put( 60,5){\line(0,1){30}} \put( 60,45){\line(0,1){35}}
\put(-70,0){\line(1,0){5}} \put(-55,0){\line(1,0){150}} \put(105,0){\line(1,0){5}}
\put(-60,-10){\line(0,1){20}} \put( 100,-10){\line(0,1){20}}
\put(-30,-40){\line(1,0){20}} \put( 50,-40){\line(1,0){5}} \put( 65,-40){\line(1,0){5}}
\put(-20,-35){\line(0,1){30}} 
\put(-20,-50){\line(0,1){5}} \put( 60,-50){\line(0,1){45}}
\put(-45,-22){$R_0$} \put(15,-22){$R_1$} \put(15,18){$R_2$} \put(75,18){$R_3$} \put(75,58){$R_4$}
\put(-18,-22){$e_1$} \put(75,33){$e_2$}  \put(-9,3){$f_1$}
\put(60,0){\circle*{4}}\put(62,-7){$v_0$}
\end{picture}
}
\caption{Case III (2)}\label{fig:caseIII(2)}
\end{figure}

\subsection*{Case III}
If \textbf{\it Case I\/} and  \textbf{\it Case II\/} do not occur in $D$, then, as illustrated in Figure~\ref{fig:caseIII}, there exist only two distinct patterns for the types of the edges incident to a non-alternating edge $f_1$ where two regions $R_1$ and $R_2$ are incident.
In the first pattern, we perform edge contractions to make $R_0$ into a triangular region and then isotope $e_1$ to eliminate the region $R_0$. We do further edge contractions to make $R_1$ into a triangular region and then isotope $e_2$ to eliminate $R_1$. If $v_0$ is the cut-point of the knot-spoke diagram just obtained, then applying Proposition~\ref{prop:main tool} we can do the edge contractions all over again so that the elimination of $R_0$ and $R_1$ is done by isotoping $e_1$ and $e_2$ without making $v_0$ into a cut-point unless there exists a simple closed curve $S$ meeting $D$ at a point of $e_2$, at a point of $e_1$ and at a vertex of $\partial R_0$ away from the edge $e_1$. In this exceptional case we can also apply Proposition~\ref{prop:main tool} so that the first reduction of $s(D)+r(D)$ is done by an isotopy of $e_1$ and the second is done by contracting $e_2$ (or $e^\prime_2$) after it become incident to $v_0$. This contraction makes $f_1$ into a loop which can be kept in a horizontal level in order to be isotoped off. In this way we obtain a cut-point free knot-spoke diagram $D^{\prime\prime}$ satisfying equation~(\ref{eqn:s+r=c}).

By assuming that \textbf{\it Case I\/},  \textbf{\it Case II\/} and \textbf{\it Case III\/} (1) do not occur anywhere in $D$, the second pattern extends as shown in Figure~\ref{fig:caseIII(2)}. We may attempt to eliminate the regions $R_1$ and $R_3$ by isotoping $e_1$ and $e_2$ after some edge contractions. The only nontrivial case that $v_0$ becomes a cut-point occurs when there exists a simple closed curve $S$ meeting $D$ in four points including a point in $e_1$ and a point in $e_2$. In this case $R_0$ and $R_4$ turned out to be the same region and therefore the edge $f_1$ can be pulled off the two crossings and isotoped around $S$ with an extra crossing. This is impossible because $D$ is a minimal crossing diagram. This completes the proof of Theorem~\ref{thm:main}.

\section*{Acknowledgments}
The authors would like to thank Morwen Thistlethwaite for his valuable comments and help during their use of Knotscape and Alexander Stoimenow for his help to eliminate duplications from their list of prime knots with more than 16 crossings.


\begin{thebibliography}{MOT}
\bibitem{BP2000}
            Yongju Bae and Chan-Young Park,
            \emph{An upper bound of arc index of links},
            Math. Proc. Camb. Phil. Soc. \textbf{129} (2000) 491--500.
\bibitem{BG2006}
            J. A. Baldwin and W. D. Gillam,
            \emph{Computations of Heegaard.Floer knot homology},
            arXiv: math/0610167.
\bibitem{B2002}
            Elisabeta Beltrami, \emph{Arc index of non-alternating links},
            J. Knot Theory Ramifications. \textbf{11}(3) (2002) 431--444.
\bibitem{C1995}
            Peter R. Cromwell, \emph{Embedding knots and links in an open book I: Basic properties},
            Topology Appl. \textbf{64} (1995) 37--58.
\bibitem{CN1996}
            Peter R. Cromwell and Ian J. Nutt,
            \emph{Embedding knots and links in an open book II. Bounds on arc index},
            Math. Proc. Camb. Phil. Soc. \textbf{119} (1996), 309--319.
\bibitem{Gong2005}
            Jae Ho Gong, Hyunwoo Kim, Seul Ah Oh, Hyuntae Kim, Gyo Taek Jin, Hun Kim and Gye-Seon Lee,
            \emph{A study on arc index of knots (in Korean)},
            Final reports of KSA R\&E Programs of 2005 [Math.] (2006) 167--215.
\bibitem{HTW1998}
            Jim Hoste, Morwen Thistlethwaite and Jeff Weeks,
            \emph{The first 1,701,936 knots},
            Math. Intelligencer \textbf{20}(4) (1998) 33--48.
\bibitem{Jin2006}
            Gyo Taek Jin, Hun Kim, Gye-Seon Lee, Jae Ho Gong, Hyuntae Kim, Hyunwoo Kim and Seul Ah Oh,
            \emph{Prime knots with arc index up to 10},
            Intelligence of Low Dimensional Topology 2006,
            Series on Knots and Everything Book vol. 40, World Scientific Publishing Co., 65--74, 2006.
\bibitem{Jin2007}
            Gyo Taek Jin and Wang Keun Park,
            \emph{A tabulation of prime knots up to arc index 11},
            submitted to arXiv.
\bibitem{H2006}
            Hiroshi Matsuda,
            \emph{Links in an open book decomposition and in the standard contact structure},
            Proc. Amer. Math. Soc. \textbf{134}(12) (2006) 3697--3702.
\bibitem{MB1998}
            Hugh R. Morton and Elisabetta Beltrami,
            \emph{Arc index and the Kauffman polynomial},
            Math. Proc. Camb. Phil. Soc. \textbf{123}(1) (1998), 41--48.
\bibitem{Ng2006}
            Lenhard Ng, \emph{On arc index and maximal Thurston-Bennequin number},
            arXiv: math/0612356
\bibitem{N1997}
            Ian J. Nutt, \emph{Arc index and Kauffman polynomial},
            J. Knot Theory Ramifications. \textbf{6}(1) (1997) 61--77.
\bibitem{N1999}
            Ian J. Nutt,
            \emph{Embedding knots and links in an open book III. On the braid index of satellite links},
            Math. Proc. Camb. Phil. Soc. \textbf{126} (1999) 77--98.
\bibitem{rolfsen}
            Dale Rolfsen, \emph{Knots and Links}, AMS Chelsea Publishing, 2003
\bibitem{knotplot}
            Knotplot, \texttt{http://knotplot.com/}
\bibitem{knotscape}
            Knotscape, \texttt{http://www.math.utk.edu/$\sim$morwen/knotscape.html}
\bibitem{knotinfo}
            Table of Knot Invariants, \texttt{http://www.indiana.edu/$\sim$knotinfo/}

\end{thebibliography}
\end{document}